\documentclass[reqno]{amsart}
\usepackage{amssymb,amsthm,amsfonts,amstext}
\usepackage{amsmath}
\usepackage{mathptmx}       
\usepackage{helvet}         
\usepackage{courier}        
\usepackage{color}

\usepackage{amscd, epsfig,enumerate}
\usepackage{graphicx}
\usepackage[usenames,dvipsnames,svgnames,table]{xcolor}
\usepackage{mathrsfs}

\usepackage[colorlinks=true,linkcolor=blue,citecolor=magenta]{hyperref}

\makeatletter \@addtoreset{equation}{section} \makeatother

\renewcommand\thefigure{\thesection.\@arabic\c@figure}
\renewcommand\thetable{\thesection.\@arabic\c@table}
\newtheorem{theorem}{Theorem}[section]
\newtheorem{lemma}[theorem]{Lemma}
\newtheorem{proposition}[theorem]{Proposition}

\newtheorem{definition}[theorem]{Definition}
\theoremstyle{remark}
\newtheorem{remark}[theorem]{Remark}

\newcommand{\mc}[1]{{\mathcal #1}}

\newcommand{\bb}[1]{{\mathbb #1}}

\newcommand{\ZZ}{\mathbb{Z}}

\renewcommand{\leq}{\leqslant}
\renewcommand{\geq}{\geqslant}
\renewcommand{\le}{\leqslant}
\renewcommand{\ge}{\geqslant}

\newcommand{\tdel}{\tilde{\Delta}}
\newcommand{\ondos}{\mc O\big( \tfrac{1}{n^2}\big)}

\title{Interpolation process between standard diffusion and fractional diffusion}
\date{}
\author{C\'edric Bernardin}
\address{Universit\'e C\^ote d'Azur, CNRS, LJAD\\
Parc Valrose\\
06108 NICE Cedex 02, France}
\email{{\tt cbernard@unice.fr}}

\author{Patr\'{\i}cia Gon\c{c}alves}
\address{\noindent Center for Mathematical Analysis,  Geometry and Dynamical Systems \\
Instituto Superior T\'ecnico, Universidade de Lisboa\\
Av. Rovisco Pais, 1049-001 Lisboa, Portugal}
\email{patricia.goncalves@math.tecnico.ulisboa.pt}

\author{ Milton Jara}
\address{\noindent Instituto de Matem\'atica Pura e Aplicada\\ Estrada Dona Castorina 110\\ 22460-320 Rio De Janeiro, Brazil.}
\email{mjara@impa.br}

\author{Marielle Simon}
\address{\noindent Inria Lille -- Nord Europe \\ 40 avenue du Halley \\ 59650 Villeneuve d'Ascq, France\\ \emph{and} Laboratoire Paul Painlev\'e, UMR CNRS 8524 \\ Cit\'e Scientifique \\ 59655 Villeneuve d'Ascq, France}
\email{marielle.simon@inria.fr}

\begin{document}

\maketitle

\begin{abstract}
We consider a Hamiltonian lattice field model with two conserved quantities, energy and volume, perturbed by stochastic noise preserving the two previous quantities. It is known that this model displays anomalous diffusion of energy of fractional type due to the conservation of the volume \cite{BS, BGJ}. We superpose to this system a second stochastic noise conserving energy but not volume. If the intensity of this noise is of order one, normal diffusion of energy is restored while it is without effect if intensity is sufficiently small. In this paper we investigate the nature of the energy fluctuations for a critical value of the intensity. We show that the latter are described by an Ornstein-Uhlenbeck process driven by a L\'evy process which interpolates between Brownian motion and  the maximally asymmetric $3/2$-stable L\'evy process. This result extends and solves a problem left open in \cite{BGJSS}.
\end{abstract}

\section{Introduction}

Since the seminal work of Fermi-Pasta-Ulam (FPU) \cite{FerPasUla}, heat conduction in chains of oscillators has attracted a lot of attention. In one-dimensional chains, superdiffusion of energy has been observed numerically in unpinned FPU chains, which corresponds to anomalous thermal conductivity. This anomalous thermal conductivity is generally attributed to a small scattering rate for low modes, which is due to momentum conservation. When the system has a pinning potential, {\textcolor{black}{destroying the conservation of momentum}}, normal diffusion of energy is expected. In \cite{BBO1,BBO2}, it was proposed to perturb the Hamiltonian dynamics with stochastic interactions that conserve energy and momentum, like random exchanges of velocity between nearest neighbours. These models  have the advantage to be studied rigorously keeping at the same time the features of deterministic models. For linear interactions, in dimension $d \geq 3$, energy follows normal diffusion, while in dimensions $d=1,2$ energy is superdiffusive \cite{BBO2}. If a pinning potential is added to the dynamics, normal diffusivity can be proved regardless of the dimension.

In \cite{JKO2} it was proved that in dimension $d=1$, energy fluctuations follow the {\it fractional heat equation} $\partial_t u = -c(-\Delta)^{3/4} u$, with $c>0$. As mentioned above, in the presence of a pinning potential, energy fluctuations follow the usual heat equation $\partial_t u = D\; \Delta u$, where $D>0$ is the diffusion coefficient.  Our goal is to provide a crossover between these two universality classes, aiming for a better understanding of the origin of the superdiffusivity of the energy in one-dimensional chains. In particular, we aim to clarify the role of the conservation of momentum.

The stochastic chains considered in \cite{BBO2} have {\it three} conserved quantities: the energy, the momentum and the {\it stretch} of the chain. Since we are interested in the role of the conservation of momentum, for simplicity we will consider a Hamiltonian lattice field model introduced by Bernardin and Stoltz \cite{BS}, which has only two conserved quantities, see Section \ref{model}, but which displays similar superdiffusion features. We call these conserved quantities {\it energy} and {\it volume}. 
In \cite{BS} the authors add to the deterministic dynamics an energy and volume conservative Poissonian noise, which is discrete in nature. Here we consider instead a conservative Brownian noise, for a reason that will be explained ahead. In \cite{BGJ} a similar result to \cite{JKO2} has been obtained by different techniques for these models. 

Let $n \in \bb N$ be a scaling parameter, which represents the inverse mesh of the stochastic chain. We add to the dynamics a second stochastic interaction that conserves only the energy and we scale down the strength of this second interaction by $\frac{a}{n}$, with $a>0$. We prove that energy fluctuations follow an evolution equation of the form \[\partial_t u = \mc L_a u,\] where $\mc L_a$ has the Fourier representation
\[
\widehat{\mc L_a}(k) = -\frac{4\pi^2 k^2}{\sqrt{a+ 2i \pi k}}, \qquad {\color{black}k\in\bb R.}
\]
In particular, we note that $\mc L_a \to -c\Big\{ (-\Delta)^{3/4} -\nabla(-\Delta)^{1/4} \Big\}$ as $a \to 0$ and $\sqrt{a} \mc L_a \to \Delta$ as $a \to \infty$, providing in this way a crossover between anomalous and normal diffusion of energy in the model. 
Note as well that the interpolation between the fractional and normal Laplacians can be understood as an ultraviolet cut-off at modes of order $\mc{O}(a)$: low modes behave diffusively, while high modes behave superdiffusively.

In \cite{BGJSS} another version of the model of Bernardin and Stoltz of \cite{BS} was considered. An almost complete phase diagram was obtained, although the interpolating part of the diagram described here was missing there. The interested reader may verify that the methods presented in this article allow to complete the phase diagram in \cite{BGJSS} as well as to prove the results stated there to the model considered here.

\subsection*{Energy fluctuations}

Let us describe in a more precise way the main result proved in \cite{BGJ} for the model considered here. Let $\{\omega_x^n(t)\}_{x \in \bb Z} \in {\bb R}^{\bb Z}$ be the infinite dimensional diffusion process defined in Section \ref{model}. The (formal) conserved quantities of the model are the energy $\sum_{x \in \bb Z} [\omega^n_x (t)]^2$ and the volume $\sum_{x \in \bb Z} \omega^n_x (t)$. Let $\{\mu_\beta\, ;\, \beta>0\}$ be the family of Gibbs  homogeneous product measures which are invariant by the dynamics. Under $\mu_\beta$, the random variables $\{\omega_x\}_{x \in \bb Z}$ are independent centered Gaussian variables with variance $\beta^{-1}$.  The probability measure on the space of trajectories which is induced by the initial law $\mu_\beta$ and the Markov process $\{\omega_x^n(t)\}_{x \in \bb Z}$ is denoted by $\bb P_\beta$ and its corresponding expectation by $\bb E_\beta$. Define the {\it energy correlation function} as
\[
S_n(t,x) = \tfrac{\beta^2}{2}\;   \bb E_\beta\Big[\big([\omega_x^n(t)]^2-\beta^{-1}\big)\big([\omega_0^n(0)]^2-\beta^{-1}\big)\Big].
\]
We prove here the following scaling limit for $S_n(t,x)$: for any test functions $\varphi, \psi: \bb R \to \bb R$ in the usual Schwartz space $\mc S(\bb R)$,
\begin{equation}
\label{limonada}
\lim_{n \to \infty}\; \tfrac{1}{n}\sum_{x,y \in \bb Z} S_n(tn^{3/2},y-x) \varphi\big(\tfrac{x}{n}\big) \psi\big(\tfrac{y}{n}\big) = \iint_{\bb R^2} P_t(v-s) \varphi(s) \psi(v) \; ds dv,
\end{equation}
where $P_t(\cdot)$ has the Fourier representation
\[
\widehat{P_t} (k) = e^{-t \widehat{\mc L_a}(k)}, \qquad k \in \bb R.
\]
In other words, $S_n(tn^{3/2},nx)$ converges, in a weak sense, to the fundamental solution of the evolution equation $\partial_t u = \mc L_a u$. 
The case $a =0 $ is the case considered in \cite{BGJ} for the model with a Poissonian noise.  
That result is a simple consequence of a stronger scaling limit, which is the main result of this article. To state it properly let us define the {\it energy fluctuation field}  as 
\begin{equation}
\label{naranja}
\mc E_t^n (\varphi) = \tfrac{1}{\sqrt n } \sum_{x \in \bb Z} \big( \big[\omega_x^n(t)\big]^2-\beta^{-1}\big) \varphi \big(\tfrac{x}{n}\big)
\end{equation}
for  test functions $\varphi:\bb R \to\bb R$ in $\mc S(\bb R)$. We will prove that this field converges in law to the Gaussian process which is the stationary solution of the equation
 \begin{equation}
\label{manzana}
\partial_t \mc E_t = \mc L_a^{\star} \mc E_t  + \sqrt{2\beta^{-2}(-\mc S_a)} \; \nabla \mc W_t,
\end{equation}
where $\mc W_t$ is a space-time white noise, $\mc L_a^{\star}$ is the adjoint of $\mc L_a$ in $\bb L^2(\bb R)$ and $\mc S_a$ is its symmetric part given by $\mc S_a= \frac{1}{2}(\mc L_a+ \mc L^{\star}_a)$. This convergence implies the limit
\[
\lim_{n \to \infty} \bb E_\beta\big[ \mc E_t^n(\varphi) \mc E_0^n(\psi)\big] = \bb E_\beta\big[ \mc E_t(\varphi)\mc E_0(\psi)\big],
\]
which is exactly the limit stated in \eqref{limonada}.

We point out that with respect to \cite{BGJ} and \cite{BGJSS}, the model considered in this article has a Brownian noise instead of a Poissonian noise. At the level of the correlation function $S_n(t,x)$, the choice of a Poissonian or a Brownian noise does not make a sensitive difference. In particular, the method of proof in this article allows to prove \eqref{limonada} also for Poissonian noises chosen in a proper way. However, at the level of the Gaussian fluctuations, key tightness estimates do not hold for Poissonian noises due to rare events that may introduce huge discontinuities on the observables we are interested in. We believe that at the level of finite-dimensional distributions the process \eqref{manzana} still describes the scaling limit of energy fluctuations in the model with the Poissonian noise considered in \cite{BGJSS}. However, it is not clear whether the obstructions in order to prove tightness are technical or intrinsic to those kind of noises.

\subsection*{A sketch of the proof}

Our proof of the convergence of the energy fluctuation field \eqref{naranja} follows the usual scheme of convergence in law of stochastic processes: we show tightness of the processes $\mc E_t^n$ in a suitable topology, then we prove that any limit point of the sequence $\{\mc E_t^n\}_{n\in\bb N}$ satisfies a weak formulation of the equation \eqref{manzana} and then we rely on a uniqueness result for the solutions of \eqref{manzana}. 


One technical difficulty comes from what is known in the literature by the {\it replacement lemma}: it is not very difficult to write down a martingale decomposition for $\mc E_t^n$ that should heuristically converge to the martingale problem associated to $\mc E_t$. But the drift term of this martingale decomposition involves the energy current $\omega_x^n(t) \omega_{x+1}^n(t)$. This current is {\it not} a function of the energy and therefore we say that the martingale problem for $\mc E_t^n$ is {\it not} closed. To overcome that, we need to replace the current $\omega_x^n(t) \omega_{x+1}^n(t)$ by a function of the energy. This is accomplished by studying the relation between the energy fluctuations and the fluctuations of the {\it correlation field} given by
\begin{equation}
\label{frutilla}
\tfrac{1}{n^{3/4}} \sum_{x,y \in \bb Z} \big( \omega_x^n(t) \omega_y^n(t) - \delta_{x,y} \beta^{-1}\big) f\big( \tfrac{x+y}{2n}, \tfrac{|y-x|}{\sqrt n}\big),
\end{equation}
on some regular two-dimensional test function $f$. Above, $\delta_{x,y}$ is the usual indicator function that equals 1 if $x=y$ and 0 otherwise. 
Note that, at least heuristically, the energy current is given by the correlation field evaluated at the diagonal $y=x+1$.
The introduction of this field is one of the main conceptual innovations in \cite{BGJ}. This field can be interpreted as the tensor product of the volume fluctuation field with itself. It turns out that volume fluctuations have two characteristic time scales. First, the speed of sound associated to the volume is equal to $2$, and therefore, volume fluctuations evolve in the hyperbolic time scale $tn$ following a linear transport equation. If the volume fluctuation field is modified by a Galilean transformation that drives out the transport dynamics, then it evolves in a diffusive time scale $tn^2$, following an equation of the form \eqref{manzana} with the operator $\mc L$ replaced by the usual Laplacian operator $\Delta$. In the definition of the correlation field \eqref{frutilla}, we introduced two different spatial scales. This non-homogeneous spatial scaling allows  to observe both natural time scales at once. In fact, the correlation field \eqref{frutilla} has a scaling limit in the hyperbolic time scale $tn$ given by the stationary solution of
\[
d {\color{black}\mc Z_t} = (-\partial_x+ \partial_{yy}^2 -a)\; {\color{black}\mc Z_t}\;  dt + d \mc M_t,
\]
where $\mc M_t$ is an infinite-dimensional martingale (see also Section \ref{sec:ou} for more details). 
We point out that although we do not prove neither this result\footnote{{\textcolor{black}{This result can be guessed by using the computations in Appendix \ref{subsec:beurkk} but its rigorous proof is not trivial and would require a paper by itself. The interested reader is invited to consult \cite{J} for a similar result in the context of the symmetric simple exclusion process.}}} nor anything related to it, this limiting equation was used as a guideline for the computations below. Since the energy fluctuations evolve in the superdiffusive time scale $tn^{3/2}$, the correlation field acts as a fast variable for the evolution of the energy. 

The structure of the paper is described as follows. Below we introduce the model with notations, and we state the main result of this work, namely Theorem \ref{theo}. Section \ref{sec:martingales} is devoted to the decomposition of the energy field into a martingale problem, using both the energy field and the correlation field. In Sections \ref{sec:tight} and \ref{sec:charac}, we prove, respectively, tightness of the processes and  characterization of their limit points, for establishing the convergence. Appendix \ref{app:levyA} collects some results on the L\'evy operator $\mc L$, while in Appendices \ref{app} and \ref{QV_convergence} we gather all technical details used along the proof. 

\section{Preliminaries}

\subsection{The model}
\label{model}
In this section we define the BS model (as introduced in \cite{BS}) with continuous noises. For that purpose we need to introduce two real parameters: $\lambda >0$ and $\gamma_n >0$, the latter depending on a scale parameter $n\in\bb N$.  Let us consider a system of diffusions evolving on the state space $\Omega := \bb R^\bb Z$, in the time scale $n^{3/2}$, and generated by the operator $n^{3/2} \mc L_n$, where $\mc L_n$ is decomposed as the sum $\mc L_n = \mc A + \lambda \mc S_{1}+ \gamma_n \mc S_{2}$, where 
\begin{align*}
\mc A &= \sum_{x \in \bb Z} \big( \omega_{x+1} - \omega_{x-1}\big) \tfrac{\partial}{\partial \omega_x} \\
\mc S_{1}&  = \sum_{x \in \bb Z} (\mathcal{X}_{x} \circ \mathcal{X}_{x} ), \qquad  \mc S_2 = \sum_{x \in \bb Z} (\mathcal{Y}_x \circ \mathcal{Y}_x),
\end{align*}
and the family of operators $\{\mc X_x, \mc Y_x\}_{x\in \bb Z}$ is given by
\begin{align*}
\mc X_x & = (\omega_{x+1}-\omega_x) \tfrac{\partial}{\partial \omega_{x-1}}+(\omega_x-\omega_{x-1}) \tfrac{\partial}{\partial \omega_{x+1}}+(\omega_{x-1}-\omega_{x+1}) \tfrac{\partial}{\partial \omega_x},
\\
\mc Y_x& = \omega_{x+1} \tfrac{\partial}{\partial \omega_x} -\omega_x \tfrac{\partial}{\partial \omega_{x+1}}.
\end{align*}
{\color{black}{The generator ${\mc A}$ is the generator corresponding to the infinite system of coupled ODE's $d\omega_{x} (t) = (\omega_{x+1} (t)- \omega_{x-1} (t) ) \, dt, \; x \in \ZZ$. A simple change of variables \cite{BS} shows that it is equivalent to the dynamics generated by an infinite system of coupled harmonic oscillators. With this change of variables, $\omega_x$ represents either the momentum of a particle or the interdistance between two nearest neighbor particles. The diffusion operator $\mc X_x$ is nothing but the generator of a Brownian motion on the circle $\{ (\omega_{x-1}, \omega_x,\omega_{x+1}) \in \bb R^3\,;\, \omega_{x-1}^2 + \omega_{x}^2+\omega_{x+1}^2 =1, \;  \omega_{x-1} + \omega_{x}+\omega_{x+1} =0\}$ while $\mc Y_x$ is the generator of a Brownian motion on the circle $\{ (\omega_x,\omega_{x+1}) \in \bb R^2\,;\,\omega_{x}^2+\omega_{x+1}^2 =1\}$.}}

We call \textit{energy} the formal quantity $\sum_{x} [\omega_x]^2$ and \textit{volume} the formal quantity $\sum_{x}\omega_x$. The Liouville operator ${\mc A}$ as well as the noise $\mc S_1$ conserves both energy and volume, while the operator $\mc S_2$ conserves only energy. 
We assume that the strength of the second noise scales as\begin{equation} \label{eq:gamma} \gamma_n = \tfrac{a}{n}\end{equation} for some $a >0$. We emphasize that one could easily treat the general case $\gamma_n=\frac{a}{n^b}$, $b\ge 0$, as in \cite{BGJSS}, using the same methods as in this paper, but we chose here to focus on the most interesting case $b=1$ where the interpolation happens. 

 The Markov process generated by the accelerated operator $n^{3/2} \mc L_n$ is denoted by $\omega^n(t) = \{\omega_x^n(t)\}_{x \in \bb Z}$. This diffusion has a family $\{\mu_\beta\; ;\; \beta >0\}$ of 
 invariant measures given by the \textit{Gibbs homogeneous product measures}
\[
\mu_\beta(d\omega) = \prod_{x \in \bb Z} \sqrt{\tfrac{\beta}{2\pi}} \exp\big(- \tfrac{\beta \omega_x^2}{2}\big) \;d\omega_x.
\] 
Here $\beta$ represents the inverse temperature, and we denote by $\langle \varphi \rangle_\beta$ the average of $\varphi: \Omega \to \bb R$ with respect to $\mu_\beta$. 

The law of the process $\{\omega^n_x(t)\; ;  \; t\geqslant 0\}_{x\in\bb Z}$ starting from the invariant measure $\mu_\beta$ is denoted by $\bb P_\beta$, and the expectation with respect to $\bb P_\beta$ is denoted by  $\bb E_\beta$. Note that under $\mu_\beta$, the averaged energy per site equals $\langle \omega_x^2\rangle_\beta=\beta^{-1}$, and the averaged volume per site equals $\langle \omega_x\rangle_\beta=0$. 

\subsection{Fluctuation fields}\label{ssec:ff}

From now on, the Markov process $\{\omega_x^n(t)\; ; \; t \geqslant 0\}_{x\in\bb Z}$ is considered starting from $\mu_\beta$.
The {\it energy fluctuation field}  is defined as the distribution-valued process $\mc E_t^n$ given by
\begin{equation}\label{eq:ener}
\mc E_t^n(\varphi)  = \tfrac{1}{\sqrt n} \sum_{x \in \bb Z} \big(\big[\omega_x^n(t)\big]^2 - \beta^{-1}\big) \varphi\big(\tfrac{x}{n} \big)
\end{equation}
for any $ \varphi: \bb R \to \bb R$ in the usual Schwartz space $\mc S(\bb R)$ of test functions. For fixed $t$ and $ \varphi$, the random variables $\mc E_t^n( \varphi)$ satisfy a central limit theorem: they converge to a centered normal random variable of variance $2 \beta^{-2} \| \varphi\|_{2}^2$, where $\|\cdot\|_{2}$ denotes the usual norm of the Hilbert space $\bb L^2(\bb R)$.

Our main goal is to obtain a convergence result for the $\mc S'(\bb R)$-valued process $\{\mc E_t^n\; ; \; t\geqslant 0\}$. It turns out that the analysis of the {\it correlation field} 
\begin{equation}
\tfrac{1}{n^{3/4}} \sum_{x, y \in \bb Z} \big( \omega_x^n(t) \omega_y^n(t) - \delta_{x,y}\;\beta^{-1} \big) f\big( \tfrac{x+y}{2n}, \tfrac{|y-x|}{\sqrt n}\big)\label{eq:corr}
\end{equation}
will play a fundamental role on the derivation of the scaling limit of $\mc E_t^n$. Recall that $\delta_{x,y}$ is the  indicator function that equals 1 if $x=y$ and 0 otherwise, and  $f: \bb R \times \bb R_+ \to \bb R$ is a smooth function. The non-isotropic scaling is crucial in order to see the scaling limit of $\mc E_t^n$.

%

\subsection{Generalized Ornstein-Uhlenbeck equation associated to a L\'evy process} \label{sec:ou} 

First of all, let us introduce some notations: for any complex number $z \in \bb C$, we denote by $\sqrt{z}$ its principal square root, which has positive real part: if $z=re^{i\theta}$ with $r\geqslant 0$ and $\theta\in(-\pi,\pi]$, then its principal square root is $\sqrt{z}=\sqrt{r}e^{i\theta/2}$. Let also $\widehat{\psi}:\bb R \to \bb C$ be the Fourier transform of a function  $\psi \in {\bb L}^1 (\bb R)$, which is defined by
\begin{equation}
\label{eq:fourier}
\widehat{\psi}(k):=\int_{\bb R} e^{-2i\pi uk} \; \psi(u)\; du, \qquad k \in \bb R.
\end{equation}
For any $\varphi\in\mc S(\bb R)$, we define $\mathcal L \varphi$ via the action of the operator $\mc L$ on Schwartz spaces: precisely, the operator $\mc L$ acts on the Fourier transform of $\varphi$ as:
\begin{equation}
\label{eq:fourierL}
\widehat{\mc L\varphi}(k)=\frac{1}{2\sqrt{3\lambda}}\frac{(2i\pi k)^2}{\sqrt{a+i\pi k}} \widehat{\varphi}(k), \qquad k \in \bb R.
\end{equation}
This operator has nice properties, stated in the next proposition:

\begin{proposition}
 \label{prop:levy}
The operator $\mathcal L$ is the generator of a L\'evy process. It leaves the space ${\mc S} (\bb R)$  invariant, and its L\'evy-Khintchine representation  is given by
\begin{equation}
\label{eq:levy-kh}
(\mc L\varphi)(u)= \int_{\bb R} \Big[\varphi(u-y)-\varphi(u)+y\varphi'(u)\Big] \Pi_a(dy),
\end{equation} 
where $\Pi_a$ is the measure on $\bb R$ defined by
\begin{equation} \label{eq:pia}
\Pi_a(dy)=-\frac{4a^{5/2}}{\sqrt{6\lambda\pi}} \; e^{-2ay}\bigg[ \frac{3}{16(ay)^{5/2}} + \frac{1}{2(ay)^{3/2}} + \frac{1}{(ay)^{1/2}}\bigg] \mathbf{1}_{(0,+\infty)}(y) .
\end{equation}
\end{proposition}

\begin{proof}
For the sake of readability, we postpone this proof to Appendix \ref{app:levy}.
\end{proof}

Let us give here an alternative definition of $\mc L \varphi$, which will turn out to be more tractable in the forthcoming computations. We claim that $\mc L\varphi$ can equivalently be defined as follows: for any $u \in \mathbb R,$
\begin{equation} \label{eq:deriv}
( \mathcal L \varphi ) (u) = -2 \partial_u f(u,0), 
\end{equation}
where $f: \mathbb R \times \mathbb R_+ \to \mathbb R$ is the function such that its Fourier transform with respect to its first variable:
\[
F_k(v) := \int_{\mathbb R} e^{-2i \pi uk} f(u,v) du, \quad k \in \mathbb R, v \geq 0,
\]
is given by
\begin{equation} \label{FToff}
F_k(v) = - \frac{1}{4 \sqrt{3 \lambda}} \frac{(2 i \pi k) \widehat{\varphi}(k)}{\sqrt{a+i \pi k}} \exp \Big( -\sqrt{\frac{a+ i \pi k}{3 \lambda}}v \Big), \quad v \geq 0.
\end{equation}
The function $f$ defined in this way satisfies the integrability conditions
\begin{equation}
 \int_{\bb R\times \bb R^+} f^2 (u, v)\; du dv < \infty \quad \textrm{and} \quad  \int_{\bb R\times \bb R^+} \partial_vf^2 (u, v) \;dudv < \infty. 
 \label{eq:integra}
 \end{equation}
Moreover the function $f$ is solution of the Laplace equation
\begin{equation}
\begin{cases}
\big(6\lambda \partial_{vv}^2 f - \partial_u f - 2 a f\big)(u,v) =0, &\quad \text{ for } u\in\bb R, v >0,\\
12 \lambda \partial_v f(u,0) =  \varphi'(u), &\quad \text{ for } u \in \bb R.
\end{cases}
\label{eq:laplacepde}
\end{equation}
This last claim is proved  in Appendix \ref{app:alter}.

\bigskip

 Let $\mc L^\star$ be the adjoint of $\mc L$ in $\bb L^2(\bb R)$ and $\mc S:=\frac{1}{2}(\mc L+\mc L^\star)$ be its symmetric part.
Let us fix a time horizon $T>0$. We are going  to explain the meaning of a \textit{stationary solution of the infinite dimensional Ornstein-Uhlenbeck equation driven by} $\mc L$, written as follows:
\begin{equation} \label{eq:OU}
\partial_t\mc E_t=\mc L^\star \mc E_t  + \sqrt{2\beta^{-2} (-\mc S)}\; \mc W_t,
\end{equation}
where $\{\mc W_t\; ; \; t \in [0,T]\}$ is a $\mc S'(\bb R)$-valued space-time white noise.

\begin{definition}
We say that an $\mc S'(\bb R)$-valued process $\{\mc E_t\; ; \; t \in [0,T]\}$ is $\beta$-\emph{stationary} if, for any $t\in[0,T]$, the $\mc S'(\bb R)$-valued random variable $\mc E_t$ is a white noise (in space) of variance $2\beta^{-2}$, namely:  for any $\varphi \in \mc S(\bb R)$, the real-valued random variable $\mc E_t(\varphi)$ has a normal distribution of mean zero and variance $2\beta^{-2}\|\varphi\|_{2}^2$.
\end{definition}

\begin{definition}\label{def:solOU} 
We say that the  $\mc S'(\bb R)$-valued process $\{\mc E_t\;;\; t\in[0,T]\}$ is a \emph{stationary solution of} \eqref{eq:OU} if: \begin{enumerate}
\item $\{\mc E_t\; ; \; t \in [0,T]\}$ is $\beta$-stationary;
\item for any {\color{black}time differentiable function  $\varphi:[0,T]\times \bb R \to \bb R$, such that for each $t\in[0,T]$ both $\varphi_t$ and $\partial_t\varphi_t$ belong to $\mc S(\bb R)$,} the process
\[
\mc E_t(\varphi_t) - \mc E_0(\varphi_0) - \int_0^t \mc E_s\big((\partial_s+\mc L)\varphi_s\big)\; ds
\]
is a continuous martingale of quadratic variation
\[
2\beta^{-2} \int_0^t \int_{\bb R} \varphi_s(u) (- \mc S \varphi_s)(u) \; duds.
\]
\end{enumerate}
\end{definition}
%
%

Thanks to the fact that $\mc L$ is the generator of a L\'evy process, the same argument used in \cite[Appendix B]{gj2015} can be worked out here to prove the uniqueness of such solutions:

\begin{proposition}[\cite{gj2015}] \label{prop:uniq}
Two stationary solutions of \eqref{eq:OU} have the same distribution.
\end{proposition}

Let us denote by $\mc C([0,T],\mc S'(\bb R))$ the space of continuous functions from $[0,T]$ to $\mc S'(\bb R)$.
Roughly speaking, the main result of this work states that the energy fluctuations described by $\mc E_t^n$ (defined in \eqref{eq:ener}) satisfy an approximate martingale problem, which, in the limit $n\to \infty$ becomes the martingale characterization of the limiting process described in Definition \ref{def:solOU}. It can be precisely formulated as follows: 

\begin{theorem}\label{theo}
The sequence of processes $\{\mc E_t^n\; ; \; t \in [0,T]\}_{n\in\bb N}$ converges in law, as $n\to\infty$, with respect to the weak topology of $\mc C([0,T], \mc S'(\bb R))$, to the stationary solution of the infinite-dimensional Ornstein-Uhlenbeck process given by \eqref{eq:OU}.
\end{theorem}

The proof of Theorem \ref{theo} follows from two steps:
\begin{enumerate}
\item We prove in Section \ref{sec:tight} that the sequence $\{\mc E_t^n\; ; \; t\in[0,T]\}_{n\in\bb N}$ is tight.
\item We characterize all its limit points in Section \ref{sec:charac} by  means of a martingale problem.
\end{enumerate}
First, we need to do an investigation of the fluctuation field $\mc E_t^n$, and of the discrete martingale problem that it satisfies. 

\section{Martingale decompositions}\label{sec:martingales}

In this section we fix $ \varphi\in\mc S(\bb R)$.  Let $f:\bb R \times \bb R_+ \to\bb R$ be as in Section \ref{sec:ou}. Let us introduce the time dependent \textit{\color{black}bidimensional field}, defined   as $\mc C_t^n(f):=\mc C(f)(\omega^n(t))$ with 
\[
\mc C(f)(\omega):= \tfrac{1}{n}\;\sum_{x, y \in \bb Z} \big( \omega_x \omega_y - \delta_{x,y}\;\beta^{-1}\big)f_{x,y}^n,\]
 where, for any $x,y \in \bb Z$,
\begin{equation} \label{eq:fxy} f_{x,y}^n:= f\big( \tfrac{x+y}{2n}, \tfrac{|y-x|}{\sqrt n}\big). 
\end{equation}
Note that, for any sufficiently regular square-integrable function $f$, since under $\mu_{\beta}$ the variables $\{\omega_x\}_{x\in\bb N}$ are independent and centered Gaussian, we have, by an application of the Cauchy-Schwarz inequality, that
\begin{equation}\label{eq:normC}
\bb E_\beta\Big[\big(\mc C_t^n(f)\big)^2\Big] \leqslant \tfrac{C(\beta)}{n^2} \sum_{x,y\in\bb Z} (f_{x,y}^n)^2 \xrightarrow[n\to\infty]{}0.
\end{equation}
\subsection{Martingale decomposition for the energy}\label{ssec:energy2}
We first need to define the two discrete operators  $\nabla_n$ and $\Delta_n$,  acting on $\varphi \in \mc S(\bb R)$ as follows:
for any $x\in \bb Z$ let \[
\nabla_n\varphi(\tfrac{x}{n}):=n\big\{\varphi(\tfrac{x+1}{n})-\varphi(\tfrac{x}{n})\big\},\quad \Delta_n\varphi(\tfrac{x}{n}):=n^2\big\{\varphi(\tfrac{x+1}{n})+\varphi(\tfrac{x-1}{n})-2\varphi(\tfrac{x}{n})\big\}.
\]
From Dynkin's formula, see for example  \cite{kl}, for any $\varphi \in \mc S(\bb R)$, the process
\begin{equation}\label{eq:mart1}
\mc M^{\mc E}_{t,n}(\varphi):=\mc E_t^n(\varphi)-\mc E_0^n(\varphi) - \int_0^t n^{3/2}\mc L_n (\mc E_s^n(\varphi))\; ds
\end{equation}
is a martingale. A straightforward computation shows that 
\begin{equation}
n^{3/2} \mc L_n(\mc E_s^n(\varphi))=  -2\sum_{x\in\bb Z} \Big\{\omega_x^n(s)\omega^n_{x+1}(s)\nabla_n\varphi\big(\tfrac{x}{n}\big)\Big\} + {\color{black} \mathcal{R}_s^n(\varphi)}  \label{eq:energy1}
\end{equation}
{\color{black}where 
\begin{align*}
\mathcal{R}_s^n(\varphi)  = & \; (2\gamma_n+4\lambda)\tfrac{1}{n}\sum_{x\in\bb Z} \big([\omega_x^n(s)]^2-\beta^{-1}\big)\Delta_n\varphi\big(\tfrac{x}{n}\big) \\
&  + (2\lambda)\tfrac{1}{n} \sum_{x\in\bb Z} \big([\omega_x^n(s)]^2-\beta^{-1}\big)\Big[n^2\big\{\varphi(\tfrac{x+2}{n})+\varphi(\tfrac{x-2}{n})-2\varphi(\tfrac x n )\big\}\Big] \\
& + (2\lambda)\tfrac{1}{n} \sum_{x\in\bb Z} \omega^n_x(s)\omega^n_{x+2}(s)\Delta_n\varphi\big(\tfrac{x+1}{n}\big) \\
& - (4\lambda)\tfrac{1}{n} \sum_{x\in\bb Z} \omega^n_x(s)\omega^n_{x+1}(s)\Big(\Delta_n\varphi\big(\tfrac{x}{n}\big)+\Delta_n\varphi\big(\tfrac{x+1}{n}\big)\Big).\end{align*}}
The second term in the right hand side of \eqref{eq:energy1}, when integrated in time between $0$ and $t$ -- namely $\int_0^t \mathcal{R}_s^n(\varphi)ds$ -- is negligible in $\bb L^2(\bb P_\beta)$ as a consequence of the Cauchy-Schwarz inequality (recall that $\langle \omega_x \omega_{x+2}\omega_y\omega_{y+2}\rangle_\beta=0$ for $x\neq y$). 
Analogously, the first term in the right hand side \eqref{eq:energy1}, integrated in time, can be replaced thanks to Cauchy-Schwarz inequality, up to a vanishing error in $\bb L^2(\bb P_\beta)$, by
\[  
-2\int_0^t \sum_{x\in\bb Z}\omega_x^n(s)\omega_{x+1}^n(s) \varphi'\big(\tfrac x n \big) ds =   -\int_0^t \sum_{x\in\bb Z} \omega_x^n(s)\omega^n_{x+1}(s)24\lambda \partial_v f\big(\tfrac{x}{n},0\big) ds
\]
the last equality being a consequence of \eqref{eq:laplacepde}.
{\color{black}Therefore, we have
\begin{equation}\label{eq:ddd}\mc E_t^n(\varphi) - \mc E_0^n(\varphi) = -\int_0^t \sum_{x\in\bb Z} \omega_x^n(s)\omega^n_{x+1}(s)24\lambda \partial_v f\big(\tfrac{x}{n},0\big) ds+\mc M_{t,n}^{\mc E}(\varphi)+ \int_0^t\varepsilon_n(s)ds,\end{equation}
where $\mc M_{t,n}^{\mc E}(\varphi)$ is a martingale, whose quadratic variation will be computed in Section \ref{ssec:quadratic}. Moreover, $\varepsilon_n(t)$ satisfies two estimates: first, for any $t>0$ fixed,
\begin{equation}
\label{eq:L2vanish}
\lim_{n\to\infty} \bb E_\beta\Big[\Big(\int_0^t\varepsilon_n(s)ds\Big)^2\Big] = 0
\end{equation} and second,
\begin{equation} \label{eq:supvanish}
\lim_{n\to\infty} \sup_{t\in[0,T]}\bb E_\beta\Big[ \big|\varepsilon_n(t)\big|^2\Big] < +\infty.\end{equation} 
}
\subsection{Martingale decomposition for the correlation field}

Now let us turn to the bidimensional field $\mc C_t^n(f)$. From Dynkin's formula, for any $f:\bb R^2 \to \bb R$, the process
\begin{equation}\label{eq:mart2}
\mc M_{t,n}^{\mc C}(f):=\mc C_t^n(f)-\mc C_0^n(f) - \int_0^t n^{3/2}\mc L_n (\mc C_s^n(f)) \; ds
\end{equation}
is a martingale. The computations of Appendix \ref{app} allow us to write
\begin{align}
n^{3/2}&\mc L_n (\mc C(f))  = - \tfrac{2}{\sqrt n} \ \sum_{x\in\bb Z} \big(\omega_x^2-\beta^{-1}\big)\Big\{ \partial_{u}  f\big(\tfrac{x}{n},0\big) + \mc O\big( \tfrac{1}{\sqrt n} \big) \Big\} +{\mc O} \big(\tfrac{1}{\sqrt n}\big) \label{eq:step0} \\
&  +  \sum_{x\in\bb Z} \omega_x\omega_{x+1} \Big\{  24\lambda \partial_v f\big(\tfrac{x}{n},0\big)+ \mc O\big( \tfrac{1}{n} \big)\Big\} \label{eq:step2} \\
& +  \tfrac{4}{\sqrt n} \ \sum_{x\in\bb Z} \Big[\omega_x\omega_{x+1} \Big\{af\big(\tfrac{x}{n},0\big)\Big\} - \omega_{x+1}\omega_{x-1} \Big\{\lambda \partial^2_{vv}f\big(\tfrac{x}{n},0\big) + \mc O\big( \tfrac{1}{\sqrt n} \big)\Big\} \Big], \label{eq:step}
\end{align}
where $\mc O(\varepsilon_n)$ denotes a sequence of functions in $\bb Z$ bounded by $c\varepsilon_n$ for some finite constant $c$ that does not depend on $n$. {\color{black}Note that \eqref{eq:step2} contains the same term that we made appear above in \eqref{eq:ddd}.}

Observe that w.r.t.~the computations of Appendix \ref{app} an extra term has been introduced (precisely in the first display \eqref{eq:step0}): this term is
$$\tfrac{2\beta^{-1}}{\sqrt n} \sum_{x \in \ZZ} \partial_u  f\big(\tfrac{x}{n},0\big) = -\tfrac{\beta^{-1}}{\sqrt n} \sum_{x \in \ZZ} ({\mc L} \varphi) \big(\tfrac{x}{n}\big)$$ where the last equality follows from \eqref{eq:deriv}. We claim that this new quantity is at most of order $n^{-1/2}$.  To justify this, recall that by Proposition \ref{prop:levy} the function $h={\mc L} \varphi$ is in the Schwartz space and that its integral equals $\int_{\bb R} h(u)du ={\hat h} (0) =0$. Moreover we have
\begin{equation*}
\begin{split}
\bigg| \tfrac{1}{n} \sum_{x \in \bb Z} h\big( \tfrac{x}{n}\big) - \int_{\bb R} h(u) du \bigg|&=\bigg|\sum_{x \in \bb Z} \int_{\frac{x}{n}}^{\frac{x+1}{n}} \Big( h\big( \tfrac{x}{n}\big) - h(u)\Big) du \bigg|=\bigg|\sum_{x \in \bb Z} \int_{\frac{x}{n}}^{\frac{x+1}{n}} h'(u) \big(\tfrac{x+1}{n} - u\big) du \bigg|\\
&\le \tfrac{1}{n} \, \sum_{x \in \bb Z} \int_{\frac{x}{n}}^{\frac{x+1}{n}} |h'(u)| du =\tfrac{1}{n} \int_{\bb R} |h' (u)| du=  {\mc O} (\tfrac{1}{n}).
\end{split}
\end{equation*}
Therefore $ \sum_{x \in \ZZ} h \big(\tfrac{x}{n}\big) ={\mc O} (1)$ and the claim is proved.

Let us go one step further, and replace the local function $\omega_{x-1}\omega_{x+1}$ that appears in \eqref{eq:step} with the local function $\omega_x\omega_{x+1}$. This is the purpose of Lemma \ref{lem:vanish} below: from that result we can rewrite the time integral  as
\begin{align*}
\int_0^t n^{3/2}\mc L_n (\mc C_s^n(f))\; ds  = & - \tfrac{2}{\sqrt n} \int_0^t \sum_{x\in\bb Z} \big([\omega^n_x(s)]^2-\beta^{-1}\big)\; \partial_{u}  f\big(\tfrac{x}{n},0\big) \; ds \\
&  + 24\lambda \int_0^t \sum_{x\in\bb Z} \omega_x^n(s)\omega^n_{x+1}(s) \;  \partial_v f\big(\tfrac{x}{n},0\big)\; ds \\
& + \tfrac{4}{\sqrt n} \int_0^t   \sum_{x\in\bb Z} \omega^n_x(s)\omega^n_{x+1}(s) \big(af - \lambda \partial^2_{vv}f\big)\big(\tfrac{x}{n},0\big) \; ds + {\color{black}\int_0^t\varepsilon'_n(s)ds, }
\end{align*}
{\color{black}where $\varepsilon'_n(t)$ satisfies
the same estimates as $\varepsilon_n(t)$, namely \eqref{eq:L2vanish} and \eqref{eq:supvanish}.} 

\begin{lemma}\label{lem:vanish}
Let $\{\psi_n(x)\}_{x\in\bb Z}$ be a  real-valued sequence such that
\begin{equation}\label{eq:normfinite}
\tfrac{1}{n}\sum_{x\in\bb Z} |\psi_n(x)|^2 < +\infty.
\end{equation} Then,
\begin{equation}\label{eq:vanish2}
\lim_{n\to\infty} \bb E_\beta\Big[\Big(\int_0^t \tfrac{1}{\sqrt n} \sum_{x\in\bb Z} \psi_n(x) (\omega^n_x-\omega^n_{x-1})(s)\omega^n_{x+1}(s)\; ds\Big)^2\Big] =0.
\end{equation}
\end{lemma}

\begin{proof}
To prove the lemma we use a general inequality for the variance of additive functionals of Markov processes: we have
\begin{equation}\label{eq:ineq-1}
\bb E_\beta\Big[\Big(\int_0^t \tfrac{1}{\sqrt n } \sum_{x\in\bb Z} \psi_n(x) (\omega^n_x-\omega^n_{x-1})(s)\omega^n_{x+1}(s)\; ds\Big)^2\Big]  \leqslant C(\beta) \tfrac{t}{n^{3/2}}  \big\|\Psi\big\|^2_{[tn^{3/2}]^{-1},-1} \end{equation}
where 
\[\Psi(\omega):= \tfrac{1}{\sqrt n} \sum_{x\in\bb Z}\psi_n(x)(\omega_x-\omega_{x-1})\omega_{x+1},\]
and, for any $z >0$, 
\begin{align}
\big\|\Psi\big\|^2_{z,-1} & := \big\langle \Psi,\; \big( z-\lambda\mc S_1-\gamma_n\mc S_2\big)^{-1}\; \Psi\big\rangle_\beta   \notag \\
& = \sup_{g} \Big\{2\big\langle \Psi\;  g\big\rangle_\beta - z \big\langle g^2 \big\rangle_\beta-\big\langle g \; (-\lambda\mc S_1-\gamma_n\mc S_2)g\big\rangle_\beta\Big\}, \label{eq:-1norm}
\end{align} where  the supremum is restricted over  functions $g$ in the domain of ${\mathcal S}$. {\color{black}In order to prove \eqref{eq:ineq-1}, we first apply Lemma 3.9 of \cite{sethu}, with the  operator $-t^{-1}{\rm Id}+n^{3/2} \mc L$ and we get: 
\begin{align*}
\bb E_\beta\Big[\Big(\int_0^t  \Psi(\omega^n(s))\; ds\Big)^2\Big]  & \leqslant C(\beta)t \big\langle \Psi, \; (t^{-1}-n^{3/2}\mc L)^{-1}\; \Psi\big\rangle_\beta \vphantom{\bigg(}\\
& = \tfrac{C(\beta)t}{n^{3/2}} \big\langle \Psi, \; ([tn^{3/2}]^{-1}-\mc L)^{-1}\; \Psi\big\rangle_\beta \vphantom{\bigg(}\\
& \leqslant \tfrac{C(\beta)t}{n^{3/2}} \big\langle \Psi, \; ([tn^{3/2}]^{-1}-\mc S)^{-1}\; \Psi\big\rangle_\beta \vphantom{\bigg(}\\
& = \tfrac{C(\beta)t}{n^{3/2}} \big\|\Psi\big\|^2_{[tn^{3/2}]^{-1},-1}.\vphantom{\bigg(}
\end{align*}}
We can forget about the positive operator $(z-\lambda\mc S_1)$, and bound the  norm \eqref{eq:-1norm} as follows:
\[\big\|\Psi\big\|^2_{z,-1} \leqslant  \big\langle \Psi,\; \big(-\gamma_n\mc S_2\big)^{-1}\; \Psi\big\rangle_\beta\; .\]
One can easily check that
\[\mc S_2\big(\tfrac{1}{4}\omega_{x-1}\omega_{x+1} - \tfrac{1}{6} \omega_x\omega_{x+1}\big) = (\omega_x-\omega_{x-1})\omega_{x+1},\]
which implies that $(-\gamma_n \mc S_2)^{-1}\;\Psi$ is explicit and given by
\[\big(-\gamma_n \mc S_2\big)^{-1}\;\Psi(\omega) = \tfrac{1}{\gamma_n\; \sqrt n} \sum_{x\in\bb Z} \psi_n(x) \big[\tfrac{1}{6}\omega_x\omega_{x+1}-\tfrac{1}{4}\omega_{x-1}\omega_{x+1}\big],
\]
so that, finally,
\[\big\|\Psi\big\|^2_{z,-1} \leqslant \tfrac{C(\beta)}{\gamma_n \; n} \sum_{x\in\bb Z} |\psi_n(x)|^2.\]
Recall $\gamma_n=\frac{a}{n}$, and then after replacing the previous bound in \eqref{eq:ineq-1} we get
\[\bb E_\beta\Big[\Big(\int_0^t \tfrac{1}{\sqrt n} \sum_{x\in\bb Z} \psi_n(x) (\omega^n_x-\omega^n_{x-1})(s)\omega^n_{x+1}(s)\; ds\Big)^2\Big]  \leqslant C(\beta)\tfrac{t}{n^{3/2}} \tfrac{n}{a} \tfrac{1}{n} \sum_{x\in\bb Z} |\psi_n(x)|^2 = \mc O\big(\tfrac{1}{\sqrt n}\big),\]
which vanishes as $n\to\infty$.
\end{proof}

\subsection{Sum of the two decompositions} Combining the two decompositions \eqref{eq:mart1} and \eqref{eq:mart2} we get
\begin{align}
\mc E_t^n(\varphi)-\mc E_0^n(\varphi) = & - \int_0^t \tfrac{2}{\sqrt n} \ \sum_{x\in\bb Z} \big([\omega_x^n]^2(s)-\beta^{-1}\big)\; \partial_{u}  f\big(\tfrac{x}{n},0\big) \; ds \label{eq:fraclap}\\
& + \int_0^t \tfrac{4}{\sqrt n} \ \sum_{x\in\bb Z} \omega_x^n(s)\omega^n_{x+1}(s)\; \big(af - \lambda \partial^2_{vv}f\big)\big(\tfrac{x}{n},0\big)\; ds \label{eq:repeat}\\
& + \mc M^{\mc C}_{t,n}(f) - \big(\mc C_t^n(f)-\mc C_0^n(f)\big) + \mc M^{\mc E}_{t,n}(\varphi) +  {\color{black}\int_0^t\varepsilon_n''(s)ds}, \label{eq:martinga}
\end{align}
{\color{black}where $\varepsilon_n''(t)=\varepsilon_n(t)+\varepsilon'(t)$}. 
Note that
\[\int_0^t \tfrac{2}{\sqrt n} \ \sum_{x\in\bb Z} \big([\omega_x^n]^2(s)-\beta^{-1}\big)\; \partial_{u}  f\big(\tfrac{x}{n},0\big) \; ds = 2 \int_0^t \mc E_s^n\big(\partial_u f(\cdot,0)\big)\; ds. \]
Since the terms in (\ref{eq:repeat}) and (\ref{eq:martinga}) will be proved to vanish, as $n\to\infty$, this will permit to close the martingale equation in terms of the energy field. From \eqref{eq:normC}, the term $\big(\mc C_t^n(f)-\mc C_0^n(f)\big) $ vanishes in $\bb L^2(\bb P_\beta)$. Finally, the term \eqref{eq:repeat}, which is in the same form as \eqref{eq:energy1} {\color{black}(but of smaller order, since it is divided by $\sqrt n$)}, is treated by repeating the same procedure: 
 let $g:\bb R \times \bb R_+ \to\bb R$ be solution of the equation 
 \begin{equation}
\left\{ \begin{aligned}
&\big(6\lambda \partial_{vv}^2 g - \partial_u g - 2 a g\big)(u,v) =0,\ \; \ \qquad \text{ for } u\in\bb R, v >0,\\
&24 \lambda \partial_v g(u,0) = 4 \big(af - \lambda \partial^2_{vv}f\big)\big(u,0), \quad \text{ for } u \in \bb R,\end{aligned}\right.
\label{eq:pde2}
\end{equation}
where $f$ is given in Section \ref{sec:ou}. The function $g$ is defined by its Fourier transform w.r.t.~the first variable as it has been done to define $f$. Then, using the same computations as before, but with $\varphi^\prime (u)$ replaced by $2(af -\lambda \partial_{vv}^2 f) (u,0)$,  we get that
\begin{align}
\int_0^t \tfrac{4}{\sqrt n} \ &\sum_{x\in\bb Z} \omega^n_x(s)\omega^n_{x+1}(s)\; \big(af - \lambda \partial^2_{vv}f\big)\big(\tfrac{x}{n},0\big)\; ds  =  \int_0^t \tfrac{24\lambda}{\sqrt n} \sum_{x\in\bb Z} \omega^n_x(s)\omega^n_{x+1}(s)\; \partial_v g\big(\tfrac{x}{n},0\big)\; ds \notag \\
 = & \; \tfrac{1}{\sqrt n}\Big(\mc C_t^n(g)-\mc C_0^n(g) - \mc M^{\mc C}_{t,n}(g)\Big) \label{eq:g1}\\
&  + \int_0^t \tfrac{2}{n} \sum_{x\in\bb Z} \big([\omega_x^n(s)]^2-\beta^{-1}\big)\; \partial_u g\big(\tfrac{x}{n},0\big)\; ds \label{eq:g2}\\
& -\int_0^t \tfrac{4}{n}  \sum_{x\in\bb Z} \omega^n_x(s)\omega^n_{x+1}(s)\; \big(ag - \lambda \partial^2_{vv}g\big)\big(\tfrac{x}{n},0\big)\; ds + {\color{black} \int_0^t\varepsilon_n'''(s)ds}. \label{eq:g3}
\end{align}
Note that in \eqref{eq:g2} we introduced the extra term
\[\tfrac {2\beta^{-1}} n \sum_{x\in\bb Z} \partial_u g\big(\tfrac x n , 0 \big) \] as we did above for $f$. The same argument works here: one can prove that this additional quantity is of order at most $\mc O(\frac 1 n)$ since  $\int_{\bb R} \partial_u g(u,0)du=0$. 

From {\color{black}the Cauchy-Schwarz inequality}, both terms \eqref{eq:g2} and \eqref{eq:g3} vanish in $\bb L^2(\bb P_\beta)$, as $n\to\infty$, and give {\color{black}a contribution $\varepsilon_n'''(t)$ which also satisfies the same conditions as \eqref{eq:L2vanish} and \eqref{eq:supvanish} (note that this is the same argument used in Section \ref{ssec:energy2})}. Besides, {\color{black}from \eqref{eq:normC}}, $ \mc C_t^n(g)-\mc C_0^n(g) $ also vanishes in $\bb L^2(\bb P_\beta)$, as $n\to\infty$. Summarizing, the approximate discrete martingale equation can be written as
\begin{align}
\mc E_t^n(\varphi)-\mc E_0^n(\varphi) = & - 2 \int_0^t \mc E_s^n\big(\partial_u f(\cdot,0)\big)\; ds \notag \\ 
& + \mc M^{\mc E}_{t,n}(\varphi) + \mc M^{\mc C}_{t,n}(f) - \tfrac{1}{\sqrt n} \mc M^{\mc C}_{t,n}(g) + {\color{black}\int_0^t\overline\varepsilon_n(s)ds},  \label{eq:decompo-martingale}
\end{align} 
{\color{black}where $\overline{\varepsilon}_n(t)$ satisfies \eqref{eq:L2vanish} and \eqref{eq:supvanish}.}
In the following paragraph, by computing quadratic variations we prove that the only martingale term that will give a non-zero contribution to the limit is the one coming from the correlation field, namely $\mc M^{\mc C}_{t,n}(f)$.

\subsection{Convergence of quadratic variations}\label{ssec:quadratic}

We start by showing that the quadratic variations of the martingales $\mc M^{\mc E}_{\cdot,n}(\varphi)$, $\mc M^{\mc C}_{\cdot,n}(f)$ and $\mc M^ {\mc C}_{\cdot,n}(g)$ converge in mean, as $n\to\infty$.

\begin{lemma}\label{lem:quadvar-energy}
For any $\varphi \in \mc S(\bb R)$ and $t>0$, 
$$ \lim_{n\to\infty} \bb E_\beta\Big[\big\langle \mc M^{\mc E}_{\cdot,n}(\varphi)\big\rangle_{t}\Big] =0.$$
\end{lemma}

\begin{proof}
We have 
\begin{align}
\big\langle \mc M^{\mc E}_{\cdot,n}(\varphi)\big\rangle_{t}& =  \tfrac{n^{3/2}}{n}\int_0^t \Big[ \mc L_n (F^2)(\omega^n(s)) - 2 F (\mc L_n F)(\omega^n(s))\Big] \; ds \label{eq:m}\\
 &  = \sqrt n \int_0^t \sum_{z\in\bb Z}\Big[2 \lambda  \big\{\mc X_z(F)\big\}^2 + 2\gamma_n \big\{\mc Y_z(F) \big\}^2\Big] (\omega^n(s))\; ds  \notag
\end{align}
where $F(\omega):=\sum_{x\in\bb Z} \omega_x^2 \; \varphi_x^n$ and $\varphi_x^n:=\varphi(\frac{x}{n})$. {\color{black}Note that \eqref{eq:m} can also be written as
\[
\sqrt n \int_0^t \big(\lambda \mc Q_1(F,F)+\gamma_n \mc Q_2(F,F)\big)(\omega^n(s)),
\]
where the bilinear operators $\mc Q_i$ ($i=1,2$) are given by
\[
\mc Q_i (f,g) = \mc S_i(fg) - f \mc S_i g - g \mc S_i f.
\]
In some contexts, the bilinear form $\mc Q_i$ is called the {\it carr\'e du champ}.}
A long but simple computation {\color{black}(using Appendix \ref{app:carre})} gives that
\begin{align}
\big\langle \mc M^{\mc E}_{\cdot,n}(\varphi)\big\rangle_{t} = \sqrt {n} \int_0^t & \bigg[4\lambda \sum_{x\in\bb Z} {\color{black}\Big(}  \omega_x^n(s)\omega^n_{x+1}(s)(\varphi_{x+1}^n-\varphi_x^n) + \omega_{x}^n(s)\omega^n_{x-1}(s)(\varphi_{x}^n-\varphi_{x-1}^n) \notag \\ 
& \qquad \qquad +\omega_{x-1}^n(s)\omega^n_{x+1}(s)(\varphi_{x-1}^n-\varphi_{x+1}^n) {\color{black}\Big)^2} \notag \\
& + 4\gamma_n \sum_{x\in\bb Z}   {\color{black}\Big(} \omega_x^n(s)\omega_{x+1}^n(s) (\varphi_{x+1}^n-\varphi_x^n) {\color{black}\Big)^2}\bigg] ds. \label{eq:quadener}
\end{align}
Therefore, taking the expectation, since $\varphi \in \mc S(\bb R)$ we get
\[
\bb E_\beta\Big[\big\langle \mc M^{\mc E}_{\cdot,n}(\varphi)\big\rangle_{t}\Big] \leq t C(\beta)  (\lambda+\gamma_n)\;  \tfrac{1}{n^{3/2}} \sum_{z\in\bb Z} \big(\nabla_n\varphi\big(\tfrac{z}{n}\big)\big)^2 = \mc O\big(\tfrac 1 {\sqrt n}\big),\]
which proves the lemma.
\end{proof}

\begin{lemma}\label{lem:quadvar-corr1}
Let $f:\bb R \times \bb R_+ \to \bb R$ be as in Section \ref{sec:ou}. Then, for $t>0$
\[
\lim_{n\to\infty}\bb E_\beta\Big[\big\langle \mc M^{\mc C}_{\cdot,n}(f)\big\rangle_{t}\Big] =
2t\beta^{-2}\int_{\bb R \times \bb R^+} (8af^2+24\lambda(\partial_v f)^2)(u,v) \; dudv. 
\]
Moreover the term on the right hand side of last expression equals to 
\[
2t\beta^{-2}\int_{\bb R} \varphi(u) (-\mc S\varphi)(u) \; du. 
\]
\end{lemma}

\begin{proof} 
As before, we have
\begin{equation}
\label{eq:patmar}
\big\langle \mc M^{\mc C}_{\cdot,n}(f)\big\rangle_{t} = \tfrac{n^{3/2}}{n^2}\int_0^t \sum_{z\in\bb Z} \Big[ 2\lambda  \big\{\mc X_z(F)\big\}^2 + 2\gamma_n \big\{\mc Y_z(F) \big\}^2\Big] (\omega^n(s))\; ds,
\end{equation}
where $F(\omega):=\sum_{x,y\in\bb Z} \omega_x \omega_y \; f_{x,y}^n$ with $f_{x,y}^n$ defined in \eqref{eq:fxy}. Since the computations are a bit longer, we decompose them as follows: first, note that
\begin{align}
(\mc X_z&)(F)=  2\omega_{z+1}\omega_{z-1}\; (-f_{z+1,z+1}^n+f_{z-1,z-1}^n-f_{z-1,z}^n+f_{z,z+1}^n) \label{eq:00} \vphantom{\Big(}\\ 
& + 2\omega_{z}\omega_{z-1}\; (f_{z,z}^n-f_{z-1,z-1}^n-f_{z,z+1}^n+f_{z-1,z+1}^n) \label{eq:01}\vphantom{\Big(}\\ 
& +2\omega_{z}\omega_{z+1}\; (-f_{z,z}^n+f_{z+1,z+1}^n-f_{z-1,z+1}^n+f_{z-1,z}^n) \label{eq:02}\vphantom{\Big(}\\
& + 2\big\{\omega_{z}^2(f_{z,z+1}^n-f_{z-1,z}^n) + \omega_{z+1}^2(f_{z-1,z+1}^n-f_{z,z+1}^n) + \omega_{z-1}^2(f_{z-1,z}^n-f_{z-1,z+1}^n)\big\} \label{eq:03}\vphantom{\Big(}\\
& + 2\sum_{y\notin\{z-1,z,z+1\}} \omega_y \Big\{\omega_z ( f_{z+1,y}^n-f_{z-1,y}^n) + \omega_{z+1}  (f_{z-1,y}^n-f_{z,y}^n) \label{eq:04}\vphantom{\Big(}\\
& \qquad \qquad \qquad \qquad + \omega_{z-1}  ( f_{z,y}^n-f_{z+1,y}^n)\Big\}\vphantom{\Big(} \label{eq:05}.
\end{align}
In the last expression we consider separately two terms: the first expression involving only the coordinates $\omega_{z-1}$, $\omega_z$ and $\omega_{z+1}$ (from \eqref{eq:00} to \eqref{eq:03}) that we denote by (I), and the last remaining sum over $y \notin\{z-1,z,z+1\}$ (namely \eqref{eq:04}--\eqref{eq:05}) that we denote by (II). In order to compute  $\bb E_\beta\big[\big\langle \mc M^{\mc C}_{\cdot,n}(f)\big\rangle_{t}\big] $, we first estimate {\color{black}the $\bb P_\beta$-average of}
\begin{equation}\label{eq:firstpart}
\tfrac{n^{3/2}}{n^2} \int_0^t \sum_{z\in\bb Z} 2\lambda \big\{\mc X_z(F)\big\}^2(\omega^n(s))\: ds
\end{equation}
to which (I) contributes as
\[
\lambda t C(\beta)  \bigg\{   \tfrac{1}{n^{3/2}} \sum_{z\in\bb Z} \Big(\partial_{v}f\big( \tfrac{z}{n},0\big)\Big)^2 +  \tfrac{1}{n^{5/2}} \sum_{z\in\bb Z} \Big(\partial_{u}f\big( \tfrac{z}{n},0\big)\Big)^2\bigg\} + \mc O\big(\tfrac{1}{n^{3/2}}\big), 
\]
therefore it vanishes as $n\to\infty$. 
The  second term (II) is the only contributor to the limit. By using a Taylor expansion (see also \eqref{eq:taylor} below), one has
\begin{align}\label{eq:taylor_exp}
f_{z+1,y}^n-f_{z-1,y}^n & = -\tfrac{2}{\sqrt n} \partial_{v}f\big( \tfrac{z+y}{2n}, \tfrac{|y-z|}{\sqrt n}\big) + \mc O\big(\tfrac{1}{n}\big),\\
f_{z-1,y}^n-f_{z,y}^n & = \tfrac{1}{\sqrt n} \partial_v f\big( \tfrac{z+y}{2n}, \tfrac{|y-z|}{\sqrt n}\big)+ \mc O\big(\tfrac{1}{n}\big).
\end{align}
Therefore, in the estimate of \eqref{eq:firstpart} the second term (II) will contribute as
\[
2\lambda  t \; \langle \omega_0^2 \omega_1^2\rangle_\beta \sum_{z\in\bb Z}\sum_{ y \notin \{z-1,z,z+1\}} \tfrac{24}{n^{3/2}}  \Big( \partial_v f\big( \tfrac{z+y}{2n}, \tfrac{|y-z|}{\sqrt n}\big)\Big)^2 +  \mc O\big(\tfrac{1}{n}\big),
\]
which converges, as $n\to\infty$, to
\begin{equation}\label{eq:sum1}
48\lambda  t \beta^{-2} \int_{\bb R \times \bb R^+} \big( \partial_v f \big)^2 (u,v)\; du dv.
\end{equation}
Let us now take care of the second stochastic noise that appears with $\mc Y_z$. We have:
\begin{align*}
(\mc Y_z)(F)  = & 2\omega_z\omega_{z+1}(f_{z,z}^n - f_{z+1,z+1}^n) - 2(\omega_z^2-\omega_{z+1}^2) f_{z,z+1}^n\\
&-  2\sum_{y\notin\{z,z+1\}} \omega_y (\omega_z f_{z+1,y}^n - \omega_{z+1} f_{z,y}^n).
\end{align*}
 Recall that $\gamma_n = \frac{a}{n}$. One can check that 
 \begin{equation}\label{eq:secondpart}
\bb E_\beta\bigg[\tfrac{n^{3/2}}{n^2} \int_0^t \sum_{z\in\bb Z} 2\gamma_n \big\{\mc Y_z(F)\big\}^2(\omega^n(s))\: ds\bigg]
\end{equation}
 can be rewritten by the translation invariance of $\mu_\beta$ as
 \[
16 a t \; \langle \omega_0^2 \omega_1^2\rangle_\beta\; \tfrac{1}{n^{3/2}}  \sum_{z\in\bb Z}\sum_{ y \notin \{z-1,z,z+1\}}\big(f_{y,z}^n\big)^2 + \mc O\big(\tfrac{1}{n}\big),
 \]
 and it converges, as $n\to\infty$, to
 \begin{equation} \label{eq:sum2}
16a t \beta^{-2} \int_{\bb R \times \bb R^+}   f ^2 (u,v)\; du dv.
 \end{equation}
 As a consequence of \eqref{eq:sum1} and \eqref{eq:sum2} , we have
 \[
 \bb E_\beta\Big[\big\langle \mc M^{\mc C}_{\cdot,n}(f)\big\rangle_{t}\Big] \xrightarrow[n\to\infty]{}
2t\beta^{-2}\int_{\bb R \times \bb R^+} (8af^2+24\lambda(\partial_v f)^2)(u,v) \; dudv. 
 \]
 An explicit resolution of \eqref{eq:laplacepde} via Fourier transforms given in Appendix \ref{app:alter} easily gives
 \[ \int_{\bb R \times \bb R^+} (8af^2+24\lambda(\partial_v f)^2)(u,v) \; dudv = \int_{\bb R} \varphi(u) (-\mc S\varphi)(u) \; du\]
 which is enough to conclude.
\end{proof}

\begin{remark}
We note that by, similar computations to the ones of the previous lemma,  we can prove that 
 \[
 \bb E_\beta\Big[\big\langle \mc M^{\mc C}_{\cdot,n}(g)\big\rangle_{t}\Big] \xrightarrow[n\to\infty]{}
2t\beta^{-2}\int_{\bb R \times \bb R^+} (ag^2+3\lambda(\partial_v g)^2)(u,v) \; dudv,
 \] {\color{black}where $g$ has been defined before as the solution to \eqref{eq:pde2}.}
\end{remark}

\begin{lemma}[$\bb L^2(\bb P_\beta)$ convergence of quadratic variations]\label{strong_convergence_QV}
For $\varphi \in \mc S(\bb R)$  and  $f:\bb R\times \bb R_+\to\bb R$ as in Section \ref{sec:ou}, we have 
\begin{align}
& \lim_{n\to \infty} \bb E_\beta\bigg[\Big( \big\langle \mc M_{\cdot, n}^{\mc E}(\varphi)\big\rangle_t -\bb E_{\beta}\Big[ \big\langle \mc M_{\cdot, n}^{\mc E}(\varphi)\big\rangle_t\Big]\Big)^2\bigg]=0, \label{eq:l2e} \\
& \lim_{n\to \infty}\bb E_\beta\bigg[\Big( \big\langle \mc M_{\cdot, n}^{\mc C}(f)\big\rangle_t - \bb E_{\beta}\Big[ \big\langle \mc M_{\cdot, n}^{\mc C}(f)\big\rangle_t\Big] \Big)^2\bigg]=0, \label{eq:l2}
\end{align}
\end{lemma}
\begin{proof}The proof of this lemma is postponed to Appendix \ref{QV_convergence}. 
\end{proof}

\subsection{Conclusion}

From Lemma \ref{lem:quadvar-energy} and the remark above, we know that, for each fixed $t>0$,  the martingales  $\mc M_{t,n}^{\mc E}(\varphi)$ and $\tfrac{1}{\sqrt n}\mc M_{t,n}^{\mc C}(g)$ vanish, as $ n\to\infty$, in $\bb L^2(\bb P_\beta)$. Therefore the non vanishing terms remaining in the right hand side of the decomposition \eqref{eq:decompo-martingale} are
\begin{equation}\label{eq:decompfinal}
 - 2 \int_0^t \mc E_s^n\big(\partial_u f(\cdot,0)\big)\; ds + \mc M^{\mc C}_{t,n}(f) .
\end{equation}

\section{Tightness}\label{sec:tight}

The tightness of the sequence $\{\mc E_t^n\; ; \; t\in[0,T]\}_{n\in\bb N}$ in the space $\mc C([0,T],\mc S'(\bb R))$ is proved by  standard arguments. 

First, Mitoma's criterion \cite{mitoma} reduces the proof of tightness of distribution-valued processes to the proof of tightness  for real-valued processes. Indeed, it is enough to show tightness of the sequence $\{\mc E_t^n(\varphi)\; ; \; t\in[0,T]\}$ for any $\varphi\in\mc S(\bb R)$. According to \eqref{eq:decompo-martingale}, we are reduced to prove that the processes
\[
\{\mc E_0^n(\varphi)\}_{n\in\bb N}, \qquad \bigg\{\int_0^t \mc E_s^n(\partial_u f(\cdot,0))\; ds \; ; \; t\in[0,T]\bigg\}_{n\in\bb N}
\]
are tight, where $f:\bb R\times\bb R_+ \to \bb R$ is solution to \eqref{eq:laplacepde}. We will also prove that the martingales 
\begin{equation}\label{eq:martingales}
\{\mc M_{t,n}^{\mc E}(\varphi)\; ; \; t\in[0,T]\}_{n\in\bb N}, \
 \{\mc M_{t,n}^{\mc C}(f)\; ; \; t\in[0,T]\}_{n\in\bb N},\ \big\{\tfrac{1}{\sqrt n}\mc M_{t,n}^{\mc C}(g)\; ; \; t\in[0,T]\big\}_{n\in\bb N}
\end{equation}
are convergent and, in particular, they are tight, {\color{black}and finally that the process 
\[\Big\{\int_0^t\overline{\varepsilon}_n(s) ds\; ; \; t\in[0,T]\Big\}_{n\in\bb N}\]
is tight.}

\subsection{Tightness for $\{\mc E_0^n(\varphi)\}_{n\in\bb N}$ }  As mentioned {\color{black}at the beginning of Section \ref{ssec:ff}},  $\{\mc E_0^n(\varphi)\}_{n\in\bb N}$ converges in distribution, as $n\to\infty$, towards a centered normal random variable of variance $2\beta^{-2} \|\varphi\|_{\bb L^2(\bb R)}^2$, and in particular the sequence is tight.

\subsection{Tightness for $\big\{\int_0^t \mc E_s^n(\partial_u f(\cdot,0))\; ds \; ; \; t\in[0,T]\big\}_{n\in\bb N}$ {\color{black}and for $\big\{\int_0^t\overline{\varepsilon}_n(s) ds\; ; \; t\in[0,T]\big\}_{n\in\bb N}$}}

For these two integral terms we use the following tightness criterion:
\begin{proposition}[{\cite[Proposition 3.4]{gj2015}}] \label{prop:crit2} A sequence of processes of the form $\big\{\int_0^t X_n(s)\; ds\; ; \; t \in [0,T]\big\}_{n\in\bb N}$ is tight with respect to the uniform topology in $\mc C([0,T],
\bb R)$ if 
\[\lim_{n\to\infty} \sup_{t\in[0,T]} \bb E\big[X_n^2(t)\big] < +\infty.\]
\end{proposition}

One can easily check from the Cauchy-Schwarz inequality that
\begin{align*}
\bb E_\beta\Big[\big(\mc E_s^n(\partial_u f(\cdot,0)) \big)^2 \Big] & \leqslant \tfrac{C(\beta)}{n} \sum_{x\in\bb Z} \big(\partial_u f(\tfrac{x}{n},0)\big)^2 \xrightarrow[n\to\infty]{} C(\beta)t^2 \int_{\bb R} (\partial_u f(u,0))^2 \; du,
\end{align*}
{\color{black}and recall that $\overline{\varepsilon}_n(t)$ satisfies \eqref{eq:supvanish}.} Therefore, the criterion of Proposition \ref{prop:crit2} holds for both processes, and tightness follows.

\subsection{Convergence of martingales}
\label{sec:conv-mart}
By definition, and more precisely \eqref{eq:mart1} and \eqref{eq:mart2}, for any $n \in \bb N$ and  $\varphi\in\mc S(\bb R)$, $f$ as in Section \ref{sec:ou} and $g$ solution of \eqref{eq:pde2}, the martingales 
\[
\{\mc M_{t,n}^{\mc E}(\varphi)\; ; \; t\in[0,T]\}, \quad
 \{\mc M_{t,n}^{\mc C}(f)\; ; \; t\in[0,T]\},\quad \big\{\tfrac{1}{\sqrt n}\mc M_{t,n}^{\mc C}(g)\; ; \; t\in[0,T]\big\}
\]
are continuous in time. 
In order to prove that the sequences of martingales written in \eqref{eq:martingales} are convergent as $n\to\infty$, we use the following criterion, adapted from \cite[Theorem 2.1]{whitt} to the case of continuous processes:


\begin{proposition}\label{prop:conv-mart}
A sequence $\{\mc M_t^n\; ; \; t\in[0,T]\}_{n \in \bb N}$ of square-integrable martingales converges in distribution with respect to the uniform topology of $\mc C([0,T];\bb R)$, as $n\to\infty$, to a Brownian motion of variance $\sigma^2$ if 
for any $t \in [0,T]$, the quadratic variation $\big\langle\mc M^n\big\rangle_t$ converges in distribution, as $n\to\infty$, towards $\sigma^2 t$.
\end{proposition}

From Section \ref{ssec:quadratic} we conclude that the martingales  
\[
\{\mc M_{t,n}^{\mc E}(\varphi)\; ; \; t\in[0,T]\}_{n\in\bb N}, \quad \big\{\tfrac{1}{\sqrt n}\mc M_{t,n}^{\mc C}(g)\; ; \; t\in[0,T]\big\}_{n\in\bb N}
\]
vanish in distribution, as $n\to\infty$, and from Proposition \ref{prop:conv-mart} we conclude that the martingales
\[
 \{\mc M_{t,n}^{\mc C}(f)\; ; \; t\in[0,T]\}_{n\in\bb N}
\]
converge in distribution as $n\to\infty$ to a Brownian motion of variance 
\[2t\beta^{-2}\int_{\bb R} \varphi(u) (-\mc S\varphi)(u) \; du.\]
From this, we conclude that all the martingales  are tight.

\section{Characterization of limit points}\label{sec:charac}

From the previous section, we know that the sequence $\{\mc E_t^n\; ; \; t\in[0,T]\}_{n\in\bb N}$ is tight. Let $\{\mc E_t \; ; \; t \in[0,T]\}$ be one limit point  in $\mc C([0,T],\mc S'(\bb R))$. For simplicity, we still index the convergent subsequence by $n$.

We already know that $\{\mc E_0^n(\varphi)\}_{n\in\bb N}$ converges in distribution, as $n\to\infty$, towards a centered Gaussian random variable of variance $2\beta^{-2} \|\varphi\|_{\bb L^2(\bb R)}^2$.

For the integral term it is easy to see that the convergence  in law
\[\int_0^t \mc E^n_s(\mc L\varphi)\; ds \xrightarrow[n\to\infty]{} \int_0^t \mc E_s(\mc L\varphi)\; ds\] holds. The convergence for the martingale term has already been proved in Section \ref{sec:conv-mart}. Putting all these elements together, we conclude that, for any $\varphi \in \mc S(\bb R)$, we have
\[
\mc E_t(\varphi) = \mc E_0(\varphi) + \int_0^t \mc E_s(\mc L\varphi) \; ds + \mc M_t(\varphi),
\]
where $\mc M_t(\varphi)$ is a Brownian motion of quadratic variation \[ 2t\beta^{-2}\int_{\bb R} \varphi(u)(-\mc S\varphi)(u)\; du.\]
By Proposition \ref{prop:uniq}, the distribution of $\{\mc E_t\; ; \; t\in[0,T]\}$ is uniquely determined. We conclude that the sequence $\{\mc E_t^n \; ; \; t\in[0,T]\}_{n\in\bb N}$ has a unique limit point, and since it is tight,  it converges to this limit point. This  ends the proof of Theorem \ref{theo}.

\appendix

\section{Fourier transforms and L\'evy-Khintchine decomposition} 
\label{app:levyA}

\subsection{L\'evy-Khintchine decomposition}
\label{app:levy}
Let us first prove that ${\mc L}$ lets ${\mc S} (\bb R)$ invariant. Since the Fourier transform is a bijection from ${\mc S} (\bb R)$ into itself, it is sufficient to prove that if $\widehat{\varphi} \in {\mc S} (\bb R)$ then ${\widehat{\mc L \varphi}} \in {\mc S}(\bb R)$. Since $a>0$, the function 
$$\theta: k \in \bb R \to  \frac{(2i\pi k)^2}{\sqrt{a+i\pi k}} \in \bb C$$
is a smooth function and we have that for any $p \ge 0$, there exist constants $C_p, \alpha_p >0$  such that
\begin{equation}
\forall \; k \in \bb R, \quad |\theta^{(p)} (k)| \le C_p (1+|k|)^{\alpha_p}.
\end{equation}
Therefore, we have that ${\widehat{\mc L \varphi}} \in {\mc S}(\bb R)$.

\bigskip

Let $X$ be a random variable distributed according to the Gamma distribution $\Gamma(\frac12,1)$. More precisely, its density $f_X$ with respect to the Lebesgue measure is given by 
\[
f_X(x):=\mathbf{1}_{(0,+\infty)}(x) \frac{e^{-x}}{\sqrt{\pi x}} , \qquad x \in \bb R,
\]
and its characteristic function is 
\begin{equation} \label{eq:charac} \Phi_X(t)=\bb E[e^{itX}]=\frac{1}{\sqrt{1-it}}=\lim_{\varepsilon \to 0} \int_\varepsilon^{+\infty} \frac{e^{-x}}{\sqrt{\pi x}} e^{itx}\; dx, \qquad t \in\bb R.\end{equation}
\begin{lemma} \label{lem:identity}
For any $t\in\bb R$,
\[ H(t):=\frac{t^2}{\sqrt{1-it}}=\int_0^{+\infty} (e^{itx}-1-itx) \Pi(dx),
\]
where $\Pi(dx):=f_X''(x)\; dx$.
\end{lemma}
\begin{proof}
Note that, for $\varepsilon >0$, an integration by parts gives
\begin{align*}
\int_\varepsilon^{+\infty} f_X(x)e^{itx}\; dx 
& = \frac{1-e^{it\varepsilon}}{it}f_X(\varepsilon)-\frac{1}{it}\int_\varepsilon^{+\infty}(e^{itx}-1)f'_X(x)\; dx \\
& \xrightarrow[\varepsilon \to 0]{} -\frac{1}{it}\int_0^{+\infty}(e^{itx}-1) f'_X(x)\; dx,
\end{align*}
the last convergence holds since $f_X(\varepsilon) \simeq \frac{1}{\sqrt{\pi\varepsilon}}$ as $\varepsilon \to 0$. Therefore, we have the following identity
\begin{equation}\label{eq:first}
-it\Phi_X(t)=\int_0^{+\infty} (e^{itx}-1)f'_X(x)\; dx.
\end{equation}
A second integration by parts can now be done in the same way, and one can check that
\[
\int_\varepsilon^{+\infty} (e^{itx}-1)f'_X(x)\; dx=\Big(\frac{1-e^{it\varepsilon}}{it}+\varepsilon\Big)f_X'(\varepsilon)-\frac{1}{it}\int_\varepsilon^{+\infty} (e^{itx}-1-itx)f''_X(x)\; dx.
\]
Since $f'_X(\varepsilon)=-\frac{e^{-\varepsilon}}{\sqrt{\pi\varepsilon}}(1+\frac{1}{2\varepsilon})$, by taking the limit as $\varepsilon \to 0$ in the previous identity, using \eqref{eq:first} and recalling \eqref{eq:charac}, Lemma \ref{lem:identity} follows.
\end{proof}
The function we are interested in is the one  that appears in \eqref{eq:fourierL}, namely:
\[
\Psi_a(t):=\frac{1}{2\sqrt{3\lambda}} \frac{(2i\pi t)^2}{\sqrt{a+i\pi t}}=-\frac{2a^{3/2}}{\sqrt{3\lambda}}\; H\Big(-\frac{\pi t}{a}\Big), \qquad a>0,
\]
where $H$ is given in Lemma \ref{lem:identity}.
From that lemma we get
\begin{align*}
\Psi_a(t)&=-\frac{2a^{3/2}}{\sqrt{3\lambda}}\; \int_0^{+\infty} \Big(e^{-i \frac{\pi tx}{a}}-1+\frac{i \pi tx}{a}\Big) f_X''(x)\; dx \\
&=-\frac{4a^{5/2}}{\sqrt{3\lambda}}\; \int_0^{+\infty} (e^{-2i\pi t y} - 1 + 2i \pi t y) \; f''_X(2a y) \;dy.
\end{align*}
A simple computation gives
\[
f''_X(x)=\frac{e^{-x}}{\sqrt{\pi x}} \Big(1+\frac{1}{x}+\frac{3}{4x^2}\Big).
\]
Therefore
\[
\Psi_a(t)=\int_{0}^{+\infty}(e^{-2i\pi t y} - 1 + 2i \pi t y) \; {\Pi}_a(dy),
\]
where $\Pi_a$ has been defined in \eqref{eq:pia}. 
Proposition \ref{prop:levy} easily follows.

\subsection{Aternative definition: Fourier transformation and resolution} \label{app:alter}
Recall that  $f:\bb R \times \bb R_+\to\bb R$ is such that its Fourier transform with respect to the first variable is given by \eqref{FToff}. 
For any fixed $k\in\bb R$, the function $F_k(\cdot)$ is solution to 
\begin{equation} \label{eq:pdefourier}\left\{\begin{aligned} & 6\lambda F''_k(v)-(2a+2i\pi k) F_k(v)= 0, \qquad v \geqslant 0, \\ & 12\lambda F'_k(0) = 2i\pi k \widehat{\varphi}(k), \end{aligned} \right.\end{equation}
If we assume \eqref{eq:deriv}, one can easily check that 
\begin{equation*}
\widehat{\mc L\varphi}(k)=-4i\pi k F_k'(0)=\frac{1}{2\sqrt{3\lambda}}\frac{(2i\pi k)^2}{\sqrt{a+i\pi k}} \widehat{\varphi}(k), \qquad k \in \bb R,
\end{equation*} and therefore it coincides with \eqref{eq:fourierL}. Moreover, by inverting in Fourier space the system \eqref{eq:pdefourier}, one can easily recover the partial differential equation satisfied by $f$ and given in \eqref{eq:laplacepde}. Finally, the integrability conditions \eqref{eq:integra} follow from the Parseval identity:
\begin{align*}
\int_{\bb R \times \bb R_+} \big[8af^2+24\lambda(\partial_v f)^2\big](u,v)\; du dv & = \int_{\bb R \times \bb R_+} 8a |F_k(v)|^2 + 24\lambda |F'_k(v)|^2 \; dk dv \\
& = \int_{\bb R} \frac{|2i\pi k|^2}{2\sqrt{6\lambda}}  \frac{\sqrt{a+|a+i\pi k|}}{|a+i\pi k|} |\widehat{\varphi}(k)|^2\; dk \\ 
 & = \int_{\bb R} \widehat{\varphi}(-k) (- \widehat{\mc S\varphi})(k)\; dk \\
 & = \int_{\bb R} \varphi(u)(-\mc S \varphi)(u)\; du.
\end{align*}

\section{Aside computations} \label{app}

\subsection{The carr\'e du champ}
\label{app:carre}
Let $f,g:\Omega \to \bb R$ be local smooth functions. 
Since the operator $\mc A$ is a first-order operator, we have the Leibniz rule
\[
\mc A (fg) = f \mc A g + g \mc A f.
\]
The operators $\mc S_1$, $\mc S_2$ are second-order differential operators. Therefore, the relation above does not hold. Recall that the bilinear operators $\mc Q_i$ ($i=1,2$) are given by
\[
\mc Q_i (f,g) = \mc S_i(fg) - f \mc S_i g - g \mc S_i f.
\]
 In our situation, these {\it carr\'es des champs} have simple expressions:
\begin{align*}
\mc Q_1(f,g) & = 2\sum_{x\in \bb Z} (\mc X_x f) ( \mc X_x g),\\
\mc Q_2(f,g) & = 2\sum_{x\in \bb Z} (\mc Y_x f) ( \mc Y_x g).
\end{align*}
We will only evaluate the carr\'e du champ on pairs of functions of the form $(\omega_x, \omega_y)$. In the case of $\mc Q_1$, we have four cases. First, $\mc Q_1(\omega_x,\omega_y)=0$ if $|y-x| \geq 3$. We have that
\begin{align*}
\mc Q_1(\omega_{x-1},\omega_{x+1})
	&= 2(\mc X_x \omega_{x-1})(\mc X_x \omega_{x+1})\\
	&= 2(\omega_{x+1}-\omega_x)(\omega_x - \omega_{x-1}).
\end{align*}
Using the identity $2(a-b)(b-c) =(a-c)^2-(a-b)^2 -(b-c)^2$ we can rewrite
\[
\mc Q_1(\omega_{x-1},\omega_{x+1}) =  (\omega_{x+1}-\omega_{x-1})^2 - (\omega_{x+1}-\omega_x)^2-(\omega_x -\omega_{x-1})^2.
\]
In a similar way,
\begin{align*}
\mc Q_1(\omega_x,\omega_{x+1}) 
		& =  2(\omega_{x+1}-\omega_{x})^2 - (\omega_{x+2}-\omega_{x+1})^2 -(\omega_{x+2}-\omega_{x})^2\\
		& \quad -(\omega_{x+1}-\omega_{x-1})^2-(\omega_{x}-\omega_{x-1})^2,\\
\mc Q_1(\omega_{x},\omega_{x}) & = 2(\omega_{x+2}-\omega_{x+1})^2+2(\omega_{x+1}-\omega_{x-1})^2 +2(\omega_{x-1}-\omega_{x-2})^2.
\end{align*}
In the case of $\mc Q_2$ we have three different cases:  
\begin{align*}
&\mc Q_2(\omega_x, \omega_y) =0, \quad |y-x| \geq 2,\\
&\mc Q_2(\omega_x,\omega_{x+1})  = -2\omega_x \omega_{x+1},\\
&\mc Q_2(\omega_x,\omega_x) = 2\omega_{x-1}^2+2\omega_{x+1}^2.
\end{align*}

\subsection{The generator applied to quadratic functions}
\label{subsec:beurkk}
As mentioned before, the correlation field plays a fundamental role in the derivation of energy fluctuations. In order to see this, we need to make a very detailed study of the action of the generator $\mc L_n$ over functions of the form
\[
\sum_{x,y \in \bb Z} \omega_x \omega_y\;  q_{x,y},
\]
where $q: \bb Z^2 \to \bb R$ will be chosen within a few lines and is supposed to be symmetric: $q_{x,y}=q_{y,x}$. We have
\begin{equation}\label{eq:ipp}
\mc L_n (\omega_x \omega_y) = \omega_x \mc L_n \omega_y + \omega_y \mc L_n \omega_x + \lambda \mc Q_1(\omega_x,\omega_y) + \gamma_n \mc Q_2(\omega_x,\omega_y).
\end{equation}
Let us introduce some notation that will be useful later on. For $u: \bb Z \to \bb R$ we define ${\tilde{\nabla}} u, \tilde{\Delta} u: \bb Z \to \bb R$ as 
\[
{\tilde{\nabla}}u_x = \tfrac{1}{2} \big( u_{x+1} - u_{x-1}\big), \quad \tdel u_x = \tfrac{1}{6} \big( u_{x-2} + 2 u_{x-1} -6 u_x +2 u_{x+1} +u_{x+2}\big).
\]
One can check that
\[
\mc L_n \omega_x = 2 {\tilde{\nabla}} \omega_x + 6\lambda \tdel \omega_x - 2 \gamma_n \omega_x.
\]
For $q:\bb Z^2 \to \bb R$ define ${\bf A} q: \bb Z^2 \to \bb R$ as
\[
{\bf A} q_{x,y} = q_{x+1,y} -q_{x-1,y} +q_{x,y+1}-q_{x,y-1}.
\]
In other words,
\[
{\bf A} q_{x,y} = 2 {\tilde{\nabla}} q_{\underset{\uparrow}{\vphantom{y}x},y} + 2 {\tilde{\nabla}} q_{x,\underset{\uparrow}{y}},
\]
where the arrows indicate on which variable the ${\tilde{\nabla}}$ operator acts. Define as well ${\bf S} q: \bb Z^2 \to \bb R$ as
\[
{\bf S} q_{x,y} = 6\tdel q_{\underset{\uparrow}{\vphantom{y}x},y} + 6 \tdel q_{x,\underset{\uparrow}{y}}.
\]
Performing an integration by parts and using \eqref{eq:ipp} we have that
\begin{align*}
\mc L_n \sum_{x,y \in \bb Z} \omega_x \omega_y q_{x,y} 
		&= \sum_{x,y \in \bb Z} \omega_x \omega_y \big( -{\bf A} + \lambda {\bf S} -4 \gamma_n\big) q_{x,y} \\
		& \quad + \sum_{x,y \in \bb Z} \big(\lambda \mc Q_1(\omega_x,\omega_y) + \gamma_n \mc Q_2(\omega_x,\omega_y) \big) q_{x,y}.
\end{align*}
The second sum on the right hand side of the last identity is what we call the {\it stochastic interaction term}, since it only appears due to the stochastic nature of the dynamics. Although the first sum also depends on the stochastic noise, it can be constructed from deterministic dynamics as well. 

The computations of Section \ref{app:carre} show that
\begin{align}
& \sum_{x,y \in \bb Z}  \mc Q_1(\omega_x, \omega_y) q_{x,y} 
		= \sum_{x \in \bb Z}2 (\omega_{x+1}-\omega_{x-1})^2 \big\{ q_{x,x} + q_{x-1,x+1} - q_{x-1,x} -q_{x,x+1}\big\} \notag\\
		&   \qquad+ \sum_{x \in \bb Z}2 (\omega_{x+1}-\omega_{x})^2 \big\{ q_{x-1,x-1} + 2 q_{x,x+1} + q_{x+2,x+2}  \big\} \notag \\ 
		& \qquad - \sum_{x\in \bb Z} 2(\omega_{x+1}-\omega_x)^2 \big\{q_{x-1,x} +q_{x-1,x+1} +q_{x,x+2}+q_{x+1,x+2}\big\}, \notag\\
		& = \sum_{x\in\bb Z}2\omega_x^2 \big\{q_{x-2,x-2}+2q_{x-1,x-1}+2q_{x+1,x+1}+q_{x+2,x+2}\big\} \notag \\
		& \qquad - \sum_{x\in\bb Z} 4\omega_x^2 \big\{q_{x-1,x+1}+q_{x+1,x+2}+q_{x-2,x-1}	\big\}\notag\\
		& \qquad - \sum_{x\in\bb Z} 4\omega_x\omega_{x+1}	\big\{ q_{x-1,x-1} + 2 q_{x,x+1} +q_{x+2,x+2}  -q_{x-1,x} -q_{x-1,x+1} -q_{x,x+2}-q_{x+1,x+2}\big\}\notag\\
		& \qquad -\sum_{x \in \bb Z}4 \omega_{x+1}\omega_{x-1} \big\{ q_{x,x} + q_{x-1,x+1} - q_{x-1,x} -q_{x,x+1}\big\}, \label{eq:car1}
		\end{align}
and we also have
		\begin{equation} \label{eq:car2}
\sum_{x,y \in \bb Z} \mc Q_2(\omega_x,\omega_y) q_{x,y} 
		= \sum_{x \in \bb Z}\big\{ 2(\omega_{x-1}^2+\omega_{x+1}^2)q_{x,x} - 4  \omega_{x} \omega_{x+1} q_{x,x+1}
\big\}.
\end{equation}
Let us go on and consider now the particular choice $q_{x,y}:=f_{x,y}^n$ given in \eqref{eq:fxy}
where $f: \bb  R \times \bb R_+ \to \bb R$ is a smooth function with enough decay at infinity. 
The computations are pretty involved; we consider in this section only the \textit{linear} part 
\[
\sum_{x,y \in \bb Z} \omega_x \omega_y \big( -{\bf  A} + \lambda {\bf S}-4\gamma_n \big) f_{x,y}^n.
\]
To simplify the notation we define $f(u,v): = f(u,-v)$ for $v <0$. We call this definition {\it symmetrization}. Extending $f$ in this way, the resulting function may be no longer differentiable at $v=0$ (but it is smooth in $u$ and has left and right derivatives in $\nu$ at $0$).  Moreover, with this extension, and recalling the definition \eqref{eq:fxy} of $f_{x,y}^n$, we have:
\begin{equation}
\label{eq:symmetry}
f_{x,y}^n = f_{y,x}^ n, \qquad \text{for any } x,y \in \bb Z.
\end{equation}
We start by computing ${\bf  A} f_{x,y}^n$ and ${\bf  S} f_{x,y}^n$. 
%
Consider $(x,y)$ situated on the  upper  half-plane delimited by the diagonal $\{x=y\}$, namely: $y\geq x$.
Then, for any $i \in \bb Z$ and for any $j \ge 0$, we have 
\begin{align}
\label{eq:taylor}  
f\big(\tfrac{x+y}{2n} + \tfrac{i}{2n}, \tfrac{y-x}{\sqrt n} + \tfrac{j}{\sqrt n}\big)  &- f \big(\tfrac{x+y}{2n}, \tfrac{y-x}{\sqrt n}\big)  \\    
&=  \tfrac{j}{\sqrt n} \partial_v f\big(\tfrac{x+y}{2n}, \tfrac{y-x}{\sqrt n}\big)  + \tfrac{1}{n} \Big( \tfrac{i}{2} \partial_u + \tfrac{j^2}{2} \partial_{vv}^2\Big) f\big(\tfrac{x+y}{2n}, \tfrac{y-x}{\sqrt n}\big)
 \notag \\ & +\tfrac{1}{n^{3/2}} \Big( \tfrac{ij}{2} \partial^2_{uv} +\tfrac{j^3}{6} \partial^3_{vvv}\Big) f\big(\tfrac{x+y}{2n}, \tfrac{y-x}{\sqrt n}\big) + \mc O_{i,j}\big( \tfrac{1}{n^2}\big), \notag
\end{align}
where $\mc O_{i,j}(\frac{1}{n^2})$ represents a sequence of functions in $\bb Z^2$ bounded by $\frac{c(i,j)}{n^2}$ for some finite constant $c(i,j)$ and for any $n \in \bb N$.  In the following, we denote $\mc O(\frac{1}{n^2})$ when the sequence of functions is bounded by $\frac{c}{n^2}$ and $c$ does not depend on any index.

From now on we denote \[
\partial f_{x,y}^n = \begin{cases} \partial f \big(\tfrac{x+y}{2n}, \tfrac{y-x}{\sqrt n}\big) & \text{ if } y > x, \\
\partial f \big(\tfrac{x}{n}, 0^+\big) & \text{ if } y=x,\end{cases}
\]where $\partial$ can be any differentiate operator involving the variable $v$.
 For $x \neq y$ we have  from \eqref{eq:taylor} that 
\[
{\bf A} f_{x,y}^n 
		= \tfrac{2}{n} \partial_u f^n_{x,y} + \mc O\big( \tfrac{1}{n^2}\big).
\]
For $x=y$, the expression is different due to the symmetrization of $f$. We have that
\[ {\bf A} f_{x,x}^n
	= \tfrac{2}{n} \partial_u  f^n_{x,x}+ \mc O\big( \tfrac{1}{n^{3/2}}\big).
\] 
Note that the term of order $\mc O(\frac{1}{n})$ is the same in both expressions, the difference appears only at order $\mc O(\frac{1}{n^{3/2}})$. 

Now let us compute ${\bf S} f_{x,y}$. The lack of regularity of $f$ at $v=0$ affects the computations if $|x-y| \leq 1$. In particular, we can ensure that all the differences of the form $f^n_{x+k,y+\ell}-f^n_{x,y}$ appear in such a way that $x+k\leq y+\ell$ and $x\leq y$. With this precaution, we avoid to cross the axis $\{x=y\}$ where derivatives can have jumps due to the irregularity of $f$. For $|y-x| \geq 2$ we have
\[{\bf S} f_{x,y}^n 
=\tfrac{12}{n} \partial_{vv}^2 f^n_{x,y} + \ondos.
\]
For $y=x+1$ we write the Taylor expansion centered at $(\frac x n,0)$ as follows
\[
{\bf  S} f_{x,x+1}^n 
		= \Big(\tfrac{4}{\sqrt n} \partial_v  +\tfrac{12}{n} \partial_{vv}^2 \Big)f^n_{x,x} + \mc O\big( \tfrac{1}{n^{3/2}}\big).
\]
For $y=x$ we have 
\[
{\bf S} f_{x,x}^n 
		= \Big(\tfrac{16}{\sqrt n} \partial_v + \tfrac{12}{n} \partial_{vv}^2 \Big)
		f^n_{x,x} + \mc O\big( \tfrac{1}{n^{3/2}}\big).
\]
Putting together all the expressions computed above, and recalling \eqref{eq:gamma}, we see that
\begin{equation*}
\begin{split}
\sum_{x,y \in \bb Z} \omega_x  \omega_y \big( - {\bf A}  + \lambda {\bf S} - 4 \gamma_n\big)  f_{x,y}^n
		 &= \sum_{x,y \in \bb Z} \omega_x \omega_y \Big\{ \tfrac{2}{n} \big( -\partial_u + 6\lambda \partial_{vv}^2 - 2a\big)f^n_{x,y}+ \ondos\Big\}\\
		&  + \sum_{x \in \bb Z} \omega_x \omega_{x+1} \Big\{ \tfrac{8\lambda}{\sqrt n} \partial_v f^n_{x,x} + \mc O\big( \tfrac{1}{n^{3/2}} \big) \Big\}\\
		& + \sum_{x \in \bb Z} \omega_x^2 \Big\{ \tfrac{16\lambda}{\sqrt n} \partial_vf^n_{x,x} + \mc O\big( \tfrac{1}{n^{3/2}}\big) \Big\}.
\end{split}
\end{equation*}

\subsection{The carr\'e du champ revisited}

In this section we perform the same computations for both {\it carr\'es des champs}. It is quite easy to see from \eqref{eq:car2} that
\begin{align*}
\sum_{x,y \in \bb Z} \gamma_n \mc Q_2(\omega_x, \omega_y) f^n_{x,y} & = \sum_{x\in\bb Z} \omega_x\omega_{x+1} \Big\{ -\tfrac{4a}{n} f^n_{x,x} + \mc O\big( \tfrac{1}{n^{3/2}}\big)\Big\}  + \sum_{x\in\bb Z} \omega_x^2 \Big\{  \tfrac{4a}{n} f^n_{x,x}+ \ondos \Big\}.
\end{align*}
We now deal with $\mc Q_1$ (see \eqref{eq:car1}). First, we consider the term with $\omega_x^2$, and we write the Taylor expansion at $(\frac x n,0)$ as 
\begin{multline*}
\tfrac{1}{2}f^n_{x-2,x-2}+f^n_{x-1,x-1}+f^n_{x+1,x+1}+\tfrac{1}{2}f^n_{x+2,x+2}-f^n_{x-1,x+1}-f^n_{x+1,x+2}-f^n_{x-2,x-1} \\
= \Big(-\tfrac{4}{\sqrt n} \partial_v - \tfrac{3}{n} \partial^2_{vv} \Big)f^n_{x,x} + \mc O\big( \tfrac{1}{n^{3/2}}\big).
\end{multline*}
Then, we have the term with $\omega_x\omega_{x+1}$, and we write the Taylor expansions at $(\frac x n,0)$:
\begin{multline*}
f^n_{x-1,x-1}+ 2 f^n_{x,x+1}  + f^n_{x+2,x+2} -f^n_{x-1,x} -f^n_{x-1,x+1} -f^n_{x,x+2}-f^n_{x+1,x+2}\\
 = \Big(-\tfrac{4}{\sqrt n} \partial_v  -\tfrac4n \partial_{vv}^2 \Big)f^n_{x,x}  +\mc O\big( \tfrac{1}{n^{3/2}} \big) .
\end{multline*}
Finally, the term with $\omega_{x+1}\omega_{x-1}$ gives the Taylor expansion centered at $(\frac x n, 0)$  as: 
\[
f^n_{x,x} + f^n_{x-1,x+1} - f^n_{x-1,x} -f^n_{x,x+1}   = \tfrac{1}{n} \partial^2_{vv}f^n_{x,x} + \mc O\big( \tfrac{1}{n^{3/2}}\big). 
\]
Therefore,
\begin{align*}
\sum_{x,y \in \bb Z} \lambda \mc Q_1(\omega_x, \omega_y) f^n_{x,y} & =   \sum_{x\in\bb Z} \omega_x^2 \Big\{ \big( -\tfrac{16\lambda}{\sqrt n}\partial_v-\tfrac{12\lambda}{n}\partial^2_{vv}\big) f^n_{x,x}+ \mc O\big( \tfrac{1}{n^{3/2}} \big) \Big\} \\
& \quad +\sum_{x\in\bb Z} \omega_x\omega_{x+1} \Big\{ \big(\tfrac{16\lambda}{\sqrt n} \partial_v +\tfrac{16\lambda}{n}\partial_{vv}^2\big)f^n_{x,x} + \mc O\big( \tfrac{1}{n^{3/2}} \big) \Big\}\\
&  \quad + \sum_{x\in\bb Z} \omega_{x+1}\omega_{x-1}\Big\{ -\tfrac{4\lambda}{n}\partial^2_{vv}f^n_{x,x}+ \mc O\big( \tfrac{1}{n^{3/2}} \big) \Big\}.
\end{align*}
Putting every computation together, we obtain
\begin{align*}
\mc L_n \sum_{x,y \in \bb Z} \omega_x \omega_y f_{x,y}^n  = & \tfrac{2}{n} \sum_{x,y \in \bb Z} \omega_x \omega_y \Big\{ \big( -\partial_u + 6\lambda \partial_{vv}^2 - 2a\big) f^n_{x,y} + \mc O\big( \tfrac{1}{n} \big) \Big\} \\
& + \tfrac{4}{\sqrt n} \ \sum_{x\in\bb Z} \omega_x\omega_{x+1} \Big\{ \big(6\lambda \partial_v-\tfrac{1}{\sqrt n}(a-4\lambda\partial_{vv}^2)\big) f^n_{x,x} + \mc O\big( \tfrac{1}{n} \big) \Big\}  \\
& + \tfrac{4}{n} \ \sum_{x\in\bb Z} \omega_x^2\Big\{ \big(-3\lambda \partial^2_{vv}+a\big)f^n_{x,x} + \mc O\big( \tfrac{1}{\sqrt n} \big) \Big\}  \\
& - \tfrac{4\lambda}{n} \sum_{x\in\bb Z} \omega_{x+1}\omega_{x-1} \Big\{ \partial^2_{vv}f^n_{x,x}+ \mc O\big( \tfrac{1}{\sqrt n} \big) \Big\}.  
\end{align*}
and after simplifications
\begin{align*}
\mc L_n \sum_{x,y \in \bb Z} \omega_x \omega_y f_{x,y}^n  =& \tfrac{2}{n} \sum_{x \neq y} \omega_x \omega_y \Big\{ \big( -\partial_u + 6\lambda \partial_{vv}^2 - 2a\big) f^n_{x,y} +\mc O\big( \tfrac{1}{n} \big)\Big\} \\
& - \tfrac{2}{n} \ \sum_{x\in\bb Z} \omega_x^2\Big\{ \partial_{u}  f^n_{x,x} + \mc O\big( \tfrac{1}{\sqrt n} \big) \Big\}  \\
&  + \tfrac{24\lambda}{\sqrt n} \ \sum_{x\in\bb Z} \omega_x\omega_{x+1} \Big\{  \partial_v f^n_{x,x}+ \mc O\big( \tfrac{1}{n} \big)\Big\} \\
& +  \tfrac{4}{n} \ \sum_{x\in\bb Z} \omega_x\omega_{x+1} \Big\{(4\lambda\partial_{vv}^2-a)f^n_{x,x}\Big\} \\ & - \tfrac 4 n \ \sum_{x \in \bb Z} \omega_{x+1}\omega_{x-1} \Big\{\lambda \partial^2_{vv}f^n_{x,x}+ \mc O\big( \tfrac{1}{\sqrt n} \big)\Big\} .  
\end{align*}

\section{$\bb L^2$ convergence of  quadratic variations}
\label{QV_convergence}
In this section we prove Lemma \ref{strong_convergence_QV}. We start by showing the $\bb L^2(\bb P_\beta)$ convergence for $\big\langle \mc M_{\cdot,n}^{\mc E}(\varphi)\big\rangle_t$, namely \eqref{eq:l2e}.
Recall the explicit formula for the quadratic variation given in \eqref{eq:quadener}. By using the inequality  $(x+y)^2\leq2x^2+2y^2$ several times, we split the four terms  appearing in \eqref{eq:quadener}  and we control each one separately by using exactly the same approach. We only give the proof of the control for one of them. We start by computing the variance of
\[
\sqrt n \; \int_0^t 4\lambda\Big\{ \sum_{x\in\bb Z}   \omega_x^n (s) \omega^n_{x+1}(s)(\varphi_{x+1}^n-\varphi_x^n) \Big\}^2\; ds,
\]
where $\varphi_x^n=\varphi(\frac x n)$.
Last expression can be written as 
\begin{align}
\label{eq:imp}
\sqrt n\;  \int_0^t & 4\lambda \sum_{x,y\in\bb Z}   \omega_x^n(s)  \omega^n_{x+1}(s)\omega^n_y(s) \omega^n_{y+1} (s)\; (\varphi_{x+1}^n-\varphi_x^n)(\varphi_{y+1}^n-\varphi_y^n) ds.
\end{align} 
Note that under the equilibrium probability measure $\mu_{\beta}$ the expectation of $[\omega_{x}^n \omega^n_{x+1} \omega_y^n  \omega^n_{y+1}] (s)$ is non-zero only for diagonal terms $y=x$, so that the expectation of (\ref{eq:imp}) is equal to
\begin{align*}
4\lambda t \sqrt n \;  \sum_{x,y\in\bb Z}  \langle  \omega_0^2\omega_{1}^2\rangle_\beta(\varphi_{x+1}^n-\varphi_x^n)^2.
\end{align*} 
Define $\chi_{x,x+1}:=\omega_x^2\omega^2_{x+1}-\langle\omega_0^2\omega_1^2\rangle_{\beta}$ which are centered random variables. By stationarity and the Cauchy-Schwarz inequality, the variance of \eqref{eq:imp} is bounded by
\begin{align}
& C t^2 n\; \int_\Omega \Big(\sum_{x\in\bb Z}\chi_{x,x+1} (\varphi_{x+1}^n-\varphi_x^n)^2\Big)^2\; \mu_\beta(d\omega) \label{eq:exp1}\\
+ \; & C t^2 n\; \int_\Omega \Big(\sum_{x\neq y\in\bb Z} \omega_x\omega_{x+1}\omega_y\omega_{y+1}(\varphi_{x+1}^n-\varphi_x^n)(\varphi_{y+1}^n-\varphi_y^n) \Big)^2\; \mu_\beta(d\omega) \label{eq:exp2}
\end{align} 
for some constant $C>0$. 
First we look at the diagonal terms. Developing the square of the sum, since the variables $\chi_{x,x+1}$ and $\chi_{y,y+1}$ are correlated only if $|y-x| \le 1$, by the Cauchy-Schwarz inequality, the term \eqref{eq:exp1} can be bounded from above by 
\begin{align*}
t^2C(\beta)n \sum_{x\in\bb Z}(\varphi_{x+1}^n-\varphi_x^n)^4={\mathcal O}(n^{-2}).
\end{align*}
For the remaining term, by developing the square of the sum and using the fact that the variables $\{\omega_x\}_{x\in\bb Z}$ have mean zero and are i.i.d.~under $\mu_\beta$ we bound it from above by
\begin{align*}
t^2C(\beta)n \sum_{x, y \in\bb Z}(\varphi_{x+1}^n-\varphi_x^n)^2 (\varphi_{y+1}^n-\varphi_y^n)^2=\mathcal{O}(n^{-1}).
\end{align*}
We let the reader work out the same argument in order to finish the proof of  \eqref{eq:l2e}. 

Now we turn to $\big\langle \mc M_{\cdot,n}^{\mc C}(f)\big\rangle_t$ and we  prove \eqref{eq:l2}. Recall the explicit expression \eqref{eq:patmar} in the proof of Lemma \ref{lem:quadvar-corr1}. We use again the inequality  $(x+y)^2\leq2x^2+2y^2$ several times and we control each term separately by using exactly the same approach. We present the proof for the contribution of the term with  $\mathcal X_z$ but we note that for the term with $\mathcal Y_z$ the estimates are analogous. Recall \eqref{eq:00}-\eqref{eq:05}. We note that the most demanding terms are those coming from \eqref{eq:04} and \eqref{eq:05}. To make the exposition as simple as possible, we look only at one of these terms, which is of the form
\[
\tfrac{1}{\sqrt n}\int_0^t\sum_{z\in\bb Z}2\lambda \Big(2\sum_{y\notin\{z-1,z,z+1\}} \omega_y^n(s)  \omega_z^n(s) (f_{z+1,y}^n-f_{z-1,y}^n)\Big)^2 \; ds
\]
and can be written as
\[
\tfrac{8\lambda}{\sqrt n}\int_0^t\sum_{z\in\bb Z} [\omega_z^n(s)]^2\Big(\sum_{y\notin\{z-1,z,z+1\}} \omega_y^n (s) (f_{z+1,y}^n-f_{z-1,y}^n)\Big)^2\; ds.
\]
We sum and subtract the mean of $\big(\omega_z^n (s)\big)^2$ to write last term as
\begin{align}
&\tfrac{8\lambda}{\sqrt n}\int_0^t\sum_{z\in\bb Z}\big( [\omega_z^n(s)]^2-\beta^{-1}\big)\Big(\sum_{y\notin\{z-1,z,z+1\}} \omega_y^n (s) (f_{z+1,y}^n-f_{z-1,y}^n)\Big)^2\; ds \label{eq:last_term1}\\
&+\tfrac{8\lambda}{\sqrt n}\int_0^t\sum_{z\in\bb Z}\beta^{-1}\Big(\sum_{y\notin\{z-1,z,z+1\}} \omega_y^n (s) (f_{z+1,y}^n-f_{z-1,y}^n)\Big)^2\; ds.\label{eq:last_term2}
\end{align}
Now we estimate the variance of each term separately. 
First, we note that the mean of \eqref{eq:last_term1} is zero so that its variance is given by
\begin{align*}
\tfrac{Ct^2}{n}\int_\Omega &\; \sum_{z\in\bb Z}(\omega_z^2-\beta^{-1})\Big(\sum_{y\notin\{z-1,z,z+1\}} \omega_y (f_{z+1,y}^n-f_{z-1,y}^n)\Big)^2\\ &\times\sum_{\bar z\in\bb Z}(\omega_{\bar z}^2-\beta^{-1})\Big(\sum_{u\notin\{\bar z-1,\bar z,\bar z+1\}} \omega_{u}  (f_{\bar z+1,u}^n-f_{\bar z-1,u}^n)\Big)^2\; \mu_\beta(d\omega).\end{align*}
To bound from above this last expression, we expand the squares and use the independence of the centered random variables $\{\omega_x\}_{x \in \ZZ}$. Therefore last expectation is bounded from above by the sum of two terms, according to $z=\bar z$ and $z \ne \bar z$. The first one is
\[
\tfrac{t^2C(\beta)}{n}\sum_{z\in\bb Z}\sum_{y,u\notin\{z-1,z, z+1\}}(f_{z+1,y}^n-f_{z-1,y}^n)^2 (f_{ z+1,u}^n-f_{ z-1,u}^n)^2,
\]
which, by \eqref{eq:taylor_exp},  can be bounded from above by 
\[
\tfrac{t^2C(\beta)}{n^3}\sum_{z\in\bb Z}\Big(\sum_{y\notin\{z-1, z,z+1\}}(\partial_vf^n_{z,y})^2\Big)^2\leq\tfrac{C}{n}
\]
and vanishes as $n\to\infty.$
The second is
\[
\tfrac{Ct^2}{n}\int_\Omega \sum_{z\neq \bar z\in\bb Z}(\omega_z^2-\beta^{-1}) \omega_{\bar z}^2(f_{z+1,\bar z}^n-f_{z-1,\bar z}^n)^2(\omega_{\bar z}^2-\beta^{-1})\omega_{z}^2  (f_{\bar z+1,z}^n-f_{\bar z-1,z}^n)^2\; \mu_\beta(d\omega).
\]
Last expectation is bounded from above by
\[
\tfrac{t^2C(\beta)}{n^3}\sum_{z\neq \bar z\in\bb Z} (\partial_vf^n_{z+1,\bar z})^4\leq \tfrac{C}{n^{3/2}},
\]
and vanishes as $n\to\infty$.
Now we compute the variance of   \eqref{eq:last_term2} which, by developing the square in the sum, can be written as 
\[\tfrac{8\lambda}{\sqrt n}\int_0^t\sum_{z\in\bb Z}\beta^{-1}\sum_{y,\bar y \notin\{z-1,z,z+1\}} \omega^n_y(s) \omega^n_{\bar y}(s)(f_{z+1,y}^n-f_{z-1,y}^n)(f_{z+1,\bar y}^n-f_{z-1,\bar y}^n)\; ds.
\]
First note that its mean is given by
\[
\tfrac{8\lambda t}{\sqrt n}\sum_{z\in\bb Z}\beta^{-2}\sum_{y\notin\{z-1,z,z+1\}}  (f_{z+1,y}^n-f_{z-1,y}^n)^2,
\]
and therefore, its variance  can be bounded from above by
\begin{align*}
&\tfrac {Ct^2}{n}\; \int_\Omega\Big(\sum_{z\in\bb Z}\beta^{-1}\sum_{y\notin\{z-1,z,z+1\}} (\omega_y^2-\beta^{-1}) (f_{z+1,y}^n-f_{z-1,y}^n)^2\Big)^2\; \mu_\beta(d\omega)\\
&+\tfrac {Ct^2}{n}\; \int_\Omega\Big(\sum_{z\in\bb Z}\beta^{-1}\sum_{\substack{y \ne \bar y\\ y, {\bar y} \notin\{z-1,z,z+1\}}} \omega_y \omega_{\bar y}(f_{z+1,y}^n-f_{z-1,y}^n)(f_{z+1,\bar y}^n-f_{z-1,\bar y}^n) \Big)^2 \; \mu_\beta(d\omega).
\end{align*}
Now, the first expectation in the previous display can be bounded from above by
$$\tfrac{C(\beta) t^2}{n^3}\sum_{z, \bar z\in \bb Z}\sum_{y\notin\{z-1,z, z+1\}}(\partial_vf^n_{z,y})^2(\partial_vf^n_{\bar z,y})^2\leq \tfrac{C}{n}$$
and vanishes as $n\to\infty$; while the second one can be bounded from above by
$$\tfrac{C(\beta) t^2}{n}\sum_{y\neq \bar y}\sum_{z,\bar z}(f_{z+1, y}^n-f_{z-1, y}^n)(f_{z+1,\bar y}^n-f_{z-1,\bar y}^n)(f_{\bar z+1,\bar y}^n-f_{\bar z-1,\bar y}^n)(f_{\bar z+1, y}^n-f_{\bar z-1,y}^n)$$
which is equal to
$$\tfrac{C(\beta) t^2}{n^3}\sum_{\substack{y \ne \bar y\\ y, {\bar y} \notin\{z-1,z,z+1\}}}\Big(\sum_{z}(\partial_vf^n_{z, y})(\partial_vf^n_{z,\bar  y})\Big)^2\leq \tfrac {C}{n}$$
and vanishes as $n\to\infty.$

\section*{Acknowledgements}
This work benefited from the support of the project EDNHS  ANR-14-CE25-0011 of the French National Research Agency (ANR) and of the PHC Pessoa Project 37854WM. 
This project has received funding from the European Research Council (ERC) under  the European Union's Horizon 2020 research and innovative programme (grant agreement No 715734).

C.B.~thanks the French National Research Agency (ANR) for its support through the grant ANR-15-CE40-0020-01 (LSD). 

 P.G.~thanks  FCT/Portugal for support through the project 
UID/MAT/04459/2013.

M.J.~thanks CNPq for its support through the grant 401628/2012-4 and FAPERJ for its support through the grant JCNE E17/2012. M.J.~was partially supported by NWO Gravitation Grant 024.002.003-NETWORKS and by MathAmSud grant LSBS-2014.

M.S.~thanks  CAPES (Brazil) and IMPA (Instituto de Matematica Pura e Aplicada, Rio de Janeiro) for the post-doctoral fellowship, and  also the Labex CEMPI (ANR-11-LABX-0007-01) for its partial support.

\bibliographystyle{amsplain}

\end{document}